\documentclass[11pt]{amsart}
\usepackage[matrix,arrow]{xy}
\usepackage{enumerate,amssymb,setspace}

\theoremstyle{definition}
\newtheorem{theorem}[equation]{Theorem}
\newtheorem{lemma}[equation]{Lemma}
\newtheorem{corollary}[equation]{Corollary}
\newtheorem{example}[equation]{Example}
\newtheorem{definition}[equation]{Definition}
\newtheorem{remark}[equation]{Remark}

\makeatletter\@addtoreset{equation}{section} \makeatother

\title{Dynamic alpha-invariants of del Pezzo surfaces}

\author{Ivan Cheltsov and Jesus Martinez-Garcia}

\thanks{Throughout this paper, we assume that all considered varieties are
projective and defined over $\mathbb{C}$.}

\begin{document}

\begin{abstract}
For every smooth del Pezzo surface $S$, smooth curve
$C\in|-K_{S}|$ and $\beta\in(0,1]$, we compute the
$\alpha$-invariant of Tian $\alpha(S,(1-\beta)C)$ and prove the
existence of K\"ahler--Einstein metrics on $S$ with edge
singularities along $C$ of angle $2\pi\beta$ for $\beta$ in
certain interval. In particular we give lower bounds for the
invariant $R(S,C)$, introduced by Donaldson as the supremum
of all $\beta\in(0,1]$ for which such a metric exists.
\end{abstract}

\maketitle

\section{Introduction}
\label{section:intro}

Let $X$ be a normal variety of dimension $n\geqslant 1$, and let
$\Delta$ be an effective $\mathbb{R}$-divisor on $X$. Suppose that
$(X,\Delta)$ has at most Kawamata log terminal singularities, and
$-(K_{X}+\Delta)$ is ample. Then $(X,\Delta)$ is a log Fano
variety. Its $\alpha$-invariant can be defined as
$$
\alpha(X,\Delta)=\mathrm{sup}\left\{\lambda\in\mathbb{R}\ \left|%
\aligned
&\text{the log pair}\ \big(X, \Delta+\lambda B\big)\ \text{is log canonical}\\
&\text{for any effective $\mathbb{R}$-divisor}\ B\sim_{\mathbb{R}} -(K_{X}+\Delta)\\
\endaligned\right.\right\}\in\mathbb{R}_{>0}.%
$$

\begin{remark}
\label{remark:lct-alpha} For every effective $\mathbb{R}$-Cartier
$\mathbb{R}$-divisor $B$ on $X$, the number
$$
\mathrm{lct}(X, \Delta; B)=\mathrm{sup}\Big\{\lambda\in\mathbb{R}\ \big\vert\ \text{the log pair}\ \big(X, \Delta+\lambda B\big)\ \text{is log canonical}\Big\}%
$$
is called the \emph{log canonical threshold} of $B$
with respect to $(X,\Delta)$. Note that
$$
\alpha(X,\Delta)=\mathrm{inf}\Big\{\mathrm{lct}\big(X,\Delta;B\big)\
\big\vert \ B\ \text{is an effective $\mathbb{R}$-divisor such
that}\ B\sim_{\mathbb{R}} -(K_{X}+\Delta)\Big\}.
$$
\end{remark}

If $\Delta=0$, we denote $\alpha(X,\Delta)$ by $\alpha(X)$. Tian
introduced $\alpha$-invariants of smooth Fano varieties in \cite{T87}.
His definition  coincides with ours by \cite[Theorem~A.3]{CSD}. In \cite{T87}, Tian also
proved

\begin{theorem}[{\cite[Theorem~2.1]{T87}}]
\label{theorem:alpha} Let $X$ be a smooth Fano variety of
dimension $n$. If \mbox{$\alpha(X)>\frac{n}{n+1}$}, then $X$
admits a K\"ahler--Einstein metric.
\end{theorem}

This theorem gives the initial motivation for the study of
$\alpha(X,\Delta)$ in the case when $\Delta=0$. In fact,
$\alpha(X,\Delta)$ is also important if $\Delta\ne 0$. When
$X$ is smooth and $\mathrm{Supp}(\Delta)$ is a
smooth irreducible divisor, Theorem~\ref{theorem:alpha} has been
generalized by Jeffres, Mazzeo and Rubinstein as follows

\begin{theorem}[{\cite[Theorem 2,~Lemma 6.13]{JMR}}]
\label{theorem:alpha-beta} Let $X$ be a smooth projective variety
of dimension $n$, and let $D$ be a smooth irreducible hypersurface
in $X$. Let $\beta\in(0,1]$ and suppose that the divisor
$-(K_{X}+(1-\beta)D)$ is ample. If
\mbox{$\alpha(X,(1-\beta)D)>\frac{n}{n+1}$}, then $X$ admits a
K\"ahler--Einstein metric with edge singularities of angle
$2\pi\beta$ along $D$.
\end{theorem}

Song computed $\alpha$-invariants of smooth toric Fano varieties
in \cite[Theorem~1.1]{Song}. The same approach can be used to
obtain an explicit combinatorial formula for $\alpha(X,\Delta)$ in
the case when $X$ is toric and $\mathrm{Supp}(\Delta)$ consists of
torus-invariant divisors (cf. \cite[Lemma~5.1]{CSD}).

\begin{example}[{\cite[Remark~6.7]{CheRu}}]%
\label{example:toric} Let $L_{1}$, $L_{2}$ and $L_{3}$ be distinct
lines on $\mathbb{P}^2$ such that $\bigcap_i L_i=\emptyset$, and
let $(\beta_1,\beta_2,\beta_3)$ be any point in $(0,1]^3$. Then
$$
\alpha\Big(\mathbb{P}^2,\sum_{i=1}^3(1-\beta_{i})L_{i}\Big)=\frac{\mathrm{max}(\beta_1,\beta_2,\beta_3)}{\beta_1+\beta_2+\beta_{3}}.
$$
\end{example}

For smooth del Pezzo surfaces, $\alpha$-invariants have been
explicitly computed in \cite[Theorem~1.7]{Ch07b} (see
\cite{ComanGuedj} and \cite{Shi} for analytic approach, see
\cite{JMG2014} for a characteristic free approach). The proof of
this theorem implies

\begin{theorem}
\label{theorem:GAFA} Let $S$ be a smooth del Pezzo surface. Then
$$
\alpha(S)=\mathrm{inf}\Bigg\{\mathrm{lct}\big(S,0;B\big)\
\Big\vert \ B\in|-K_{S}|\ \text{and}\ B=\sum B_i,\ \text{where}\
B_i\cong\mathbb{P}^1\ \text{and}\ -K_{S}\cdot B_i\leqslant 3\
\forall i\Bigg\}.
$$
\end{theorem}

To apply Theorem~\ref{theorem:alpha-beta}, the divisor
$-(K_X+(1-\beta)D)$ must be ample. A natural choice for the pair
$(X,D)$ considered by Donaldson in his approach to the
Yau-Tian-Donaldson conjecture is to let $X$ be a smooth Fano
variety and let $D$ be a smooth anticanonical divisor (see
\cite{Donaldson2012}).

\begin{remark}
\label{remark:Donaldson} Let $X$ be a smooth Fano variety of
dimension $n$, and let $D$ be a smooth divisor in
$\vert-K_X\vert$. By \cite[Theorem~1.2]{Sho}, such divisor $D$
always exists when $n\leqslant 3$, which is no longer true in
general if $n\geqslant 4$ (see \cite[Example~2.12]{HoringVoisin}).
One has $\alpha(X, (1-\beta)D)=1>\frac{n}{n+1}$ for all positive
$\beta\ll 1$ (see Theorem~\ref{theorem:Berman}). In particular,
$X$ admits a K\"ahler--Einstein metric with edge singularities of
angle $2\pi\beta$ along $D$ for all positive $\beta\ll 1$ by
Theorem~\ref{theorem:alpha-beta}.
\end{remark}

A K\"ahler--Einstein metric with singularities along $D$ of angle
$2\pi$ is  a K\"ahler--Einstein metric in the usual sense. So, it
is natural to consider the following invariant introduced by
Donaldson:

\begin{definition}[{\cite{Donaldson2012}}]
Let $X$ be a smooth Fano variety, and let $D$ be a smooth divisor
in $\vert-K_X\vert$. Then $R(X,D)$ is the supremum of all
$\beta\in (0,1]$ such that $X$ admits a K\"ahler--Einstein metric
with edge singularities along $D$ of angle $2\pi\beta$.
\end{definition}

It follows from \cite{JMR} that the smooth Fano variety $X$ admits
a K\"ahler--Einstein metric with edge singularities of angle
$2\pi\beta$ along $D$ for every positive \mbox{$\beta<R(X,D)$}.

\begin{corollary}
\label{corollary:KE-Fano} Let $X$ be a smooth Fano variety, and
let $D$ be a smooth divisor in $|-K_{X}|$. Suppose that $X$ admits
a K\"ahler-Einstein metric. Then $R(X,D)=1$.
\end{corollary}

By Tian's theorem (see \cite{Ti90}), a smooth del Pezzo surface
$S$ admits a K\"ahler--Einstein metric if and only if
$S\not\cong\mathbb{F}_1$ and $K_{S}^2\ne 7$. Thus, we have

\begin{corollary}[{\cite{Ti90}}]
\label{corollary:Berman-Tian} Let $S$ be a smooth del Pezzo
surface such that $S\not\cong\mathbb{F}_1$ and $K_{S}^2\ne 7$, and
let $C$ be a smooth curve in $|-K_{S}|$. Then $R(S,C)=1$.
\end{corollary}

Unless $R(X,D)=1$, we do not know a single example for which the
invariant $R(X,D)$ is known precisely (cf.
\cite[Theorem~1.7]{LiSun}). A lower bound for $R(X,D)$ can be
found using

\begin{theorem}[{\cite{Berm}, \cite[Corollary~5.5]{OdakaSun}, \cite[Proposition~6.10]{CheRu}, \cite[Remark~6.11]{CheRu}}]
\label{theorem:Berman} Let $X$ be a smooth Fano variety of
dimension $n$, and let $D$ be a smooth divisor in $|-K_{X}|$. Let
$$
M=\left\{\aligned
&9\ \text{if}\ n=2,\\
&64\ \text{if}\ n=3,\\
&3^n(2^n-1)^n(n+1)^{n(n+2)(2^n-1)}(2n(n+1)(n+2)!)^{n-1}\ \text{if}\ n\geqslant 4.\\
\endaligned\right.
$$
Then
$1\geqslant\alpha(X,(1-\beta)D)\geqslant\mathrm{min}\{1,\frac{1}{M\beta}\}$
for every $\beta\in(0,1]$.
\end{theorem}

\begin{corollary}
\label{corollary:alpha-beta} In the assumptions and notation of
Theorem~\ref{theorem:Berman}, one has $R(X,D)\geqslant
\frac{n+1}{nM}$.
\end{corollary}

The purpose of this paper is to merge Theorem~\ref{theorem:GAFA}
with Theorem~\ref{theorem:Berman} by proving

\begin{theorem}
\label{theorem:main} Let $S$ be a smooth del Pezzo surface, let
$C$ be a smooth curve in $|-K_S|$, and let $\beta$ be a real
number in $(0,1]$. Then
$$
\alpha\big(S,(1-\beta)C\big)=\mathrm{inf}\left\{\mathrm{lct}\big(S,(1-\beta)C;\beta B\big)\left|%
\aligned &B\in|-K_{S}|\ \text{such that}\ B=C\ \text{or}\ B=\sum B_i,\\
&\text{where}\ B_i\cong\mathbb{P}^1\ \text{and}\ -K_{S}\cdot B_i\leqslant 3\ \forall i\\
\endaligned\right.\right\}.%
$$
\end{theorem}

We will prove Theorem~\ref{theorem:main} in
Section~\ref{section:proof}. In Section~\ref{section:formulas}, we
will give very explicit formulas for the invariant
$\alpha(S,(1-\beta)C)$. Instead of presenting them here, let us
consider their applications.

\begin{corollary}
\label{corollary:dP} Let $S$ be a smooth del Pezzo surface, and
let $C$ be a smooth curve in $|-K_S|$. Then $\alpha(S,
(1-\beta)C)$ is a decreasing continuous piecewise smooth function for
$\beta\in(0,1]$.
\end{corollary}

\begin{corollary}
\label{corollary:delPezzos-blow-ups} Let $S_1$ and $S_2$ be smooth
del Pezzo surfaces, let $C_1$ and $C_2$ be smooth curves in
$|-K_{S_1}|$ and $|-K_{S_2}|$, respectively. Suppose that there is
a birational morphism $f\colon S_2\to S_1$ such that $f(C_2)=C_1$.
Then $\alpha(S_1, (1-\beta)C_1)\leqslant\alpha(S_2, (1-\beta)C_2)$
for every $\beta\in(0,1]$ except the following cases:
\begin{enumerate}
\item $S_1\cong\mathbb{P}^2$, $S_2\cong\mathbb{F}_1$, and $f$ is the blow up of an inflection point of the cubic curve $C_1\subset\mathbb{P}^2$,%

\item $S_1\cong\mathbb{P}^1\times\mathbb{P}^1$, $K_{S_2}^2=7$, and $f$ is the blow up of a point in $C_1$.%
\end{enumerate}
\end{corollary}

If $S$ is a smooth del Pezzo surface such that either
$S\cong\mathbb{F}_1$ or $K_{S}^2=7$, and $C$ is a smooth curve in
$\vert-K_S\vert$, then $R(S,C)\geqslant \frac{1}{6}$ by Corollary
\ref{corollary:alpha-beta}. We improve this bound:
\begin{corollary}
\label{corollary:F1} Suppose that $S\cong\mathbb{F}_1$. Let $C$ be
a smooth curve in $|-K_S|$. Then $R(S, C)\geqslant \frac{3}{10}$.
Furthermore, if $C$ is chosen to be
 \emph{general} in $|-K_{S}|$, then $R(S,C)\geqslant \frac{3}{7}$.
\end{corollary}

\begin{corollary}
\label{corollary:dP7} Let $S$ be a smooth del Pezzo surface such
that $K_{S}^2=7$, and let $C$ be a smooth curve in $|-K_S|$. Then
$R(S, C)\geqslant \frac{3}{7}$. Furthermore, if $C$ does not pass
through the intersection point of two intersecting $(-1)$-curves
in $S$, then $R(S,C)\geqslant \frac{1}{2}$.
\end{corollary}

In \cite[Theorem~1]{Gabor2012}, Sz\'ekelyhidi proved that
$R(S,C)\leqslant\frac{4}{5}$ when $S=\mathbb{F}_1$, and
$R(S,C)\leqslant\frac{7}{9}$ when $K_{S}^2=7$ and $C$ passes
through the intersection point of two intersecting $(-1)$-curves
in $S$.

\subsection*{Acknowledgements}
Most of the results in this paper were obtained in the Ph.D. Thesis of
the second author (see \cite{JMG-thesis}). This paper contains
simplified proofs and, in some cases, completely new proofs. A
part of this work was done during a visit of both authors to the
National Center for Theoretical Sciences in Taipei. We are
grateful to Jungkai Chen for making this visit possible. The first
author thanks the Max Planck Institute for Mathematics in Bonn,
where part of this work was completed. We thank Yanir Rubinstein
and Gabor Sz\'ekelyhidi for useful discussions.

\section{Explicit formulas}
\label{section:formulas}

Let $S$ be a smooth del Pezzo surface. If $K_{S}^2\geqslant 3$,
then $-K_{S}$ is very ample. In this case, we will identify $S$
with its anticanonical image, and we will call a curve $Z\subset
S$ such that $Z\cdot (-K_S)=1,2,3$ a line, conic, cubic,
respectively. Let $C$ be a smooth curve in $|-K_S|$, and let
$\beta$ be a positive real number in $(0,1]$. Let
$$
\check{\alpha}\big(S,(1-\beta)C\big)=\mathrm{inf}\left\{\mathrm{lct}\big(S,(1-\beta)C;\beta B\big)\left|%
\aligned &B\in|-K_{S}|\ \text{such that}\ B=C\ \text{or}\ B=\sum B_i,\\
&\text{where}\ B_i\cong\mathbb{P}^1\ \text{and}\ -K_{S}\cdot B_i\leqslant 3\ \forall i\\
\endaligned\right.\right\}.%
$$
Then $\alpha(S, (1-\beta)C)\leqslant\check{\alpha}(S,
(1-\beta)C)$. Theorem~\ref{theorem:main} states that $\alpha(S,
(1-\beta)C)=\check{\alpha}(S, (1-\beta)C)$. In this section, we
will define a number $\hat{\alpha}(S, (1-\beta)C)$ such that
$\hat{\alpha}(S, (1-\beta)C)\geqslant\check{\alpha}(S,
(1-\beta)C)$. In Section~\ref{section:proof}, we will prove that
$\alpha(S, (1-\beta)C)\geqslant\hat{\alpha}(S, (1-\beta)C)$. The
latter inequality implies Theorem~\ref{theorem:main}, since
$\hat{\alpha}(S, (1-\beta)C)\geqslant\check\alpha(S,
(1-\beta)C)\geqslant\alpha(S, (1-\beta)C)$.

\subsection{Projective plane}
\label{subsection:P2} Suppose that $S\cong\mathbb{P}^2$. Then $C$
is a smooth cubic curve on $S$. Let
$$
\hat{\alpha}\big(S, (1-\beta)C\big)=\mathrm{min}\Big\{1,
\frac{1+3\beta}{9\beta},\frac{1}{3\beta}\Big\}=\left\{\aligned
& 1 \text{ for } 0<\beta\leqslant \frac{1}{6},\\
&\frac{1+3\beta}{9\beta} \text{ for } \frac{1}{6}\leqslant \beta\leqslant \frac{2}{3},\\
&\frac{1}{3\beta} \text{ for } \frac{2}{3}\leqslant \beta\leqslant 1.%
\endaligned\right.
$$
Let $P$ be an inflection point of the curve $C$, and let $T$ be
the line in $\mathbb{P}^2$ that is tangent to $C$ at the point
$P$. Then
$\hat{\alpha}(S,(1-\beta)C)\geqslant\check{\alpha}(S,(1-\beta)C)$,
since
$$
\hat{\alpha}\big(S,
(1-\beta)C\big)=\mathrm{min}\Big\{\mathrm{lct}\big(S,(1-\beta)C;\beta
C\big),\mathrm{lct}\big(S,(1-\beta)C;3\beta T\big)\Big\}.
$$

\subsection{Smooth quadric surface}
\label{subsection:quadric} Suppose that
$S\cong\mathbb{P}^1\times\mathbb{P}^1$. Let
$$
\hat{\alpha}\big(S,
(1-\beta)C\big)=\mathrm{min}\Big\{1,\frac{1+2\beta}{6\beta}\Big\}=\left\{\aligned
&1 \text{ for } 0<\beta\leqslant \frac{1}{4},\\
&\frac{1+2\beta}{6\beta} \text{ for } \frac{1}{4}\leqslant
\beta\leqslant 1.
\endaligned\right.
$$
Let $T$ be a divisor of bi-degree $(1,1)$ on $S$ that is a union
of two fibers of each projection from $S$ to $\mathbb{P}^1$.
Suppose in addition that one component of $T$ is tangent to $C$ at
some point, and another component of $T$ passes through this
point. Then
$\hat{\alpha}(S,(1-\beta)C)\geqslant\check{\alpha}(S,(1-\beta)C)$,
since
$$
\hat{\alpha}\big(S,(1-\beta)C\big)=\mathrm{min}\Big\{\mathrm{lct}\big(S,(1-\beta)C;\beta
C\big),\mathrm{lct}\big(S,(1-\beta)C;2\beta T\big)\Big\}.
$$

\subsection{First Hirzebruch surface}
\label{subsection:F1} Suppose that $S\cong\mathbb{F}_1$. Let $Z$
be the unique $(-1)$-curve in $S$, and let $F$ be the fiber of the
natural projection $S\to\mathbb{P}^1$ that passes through the
point $C\cap Z$. Then $C\sim 2Z+3F$. If $F$ is tangent to $C$ at
the point $C\cap Z$, let
$$
\hat{\alpha}\big(S,(1-\beta)C\big)=\mathrm{min}\Big\{1,\frac{1+2\beta}{8\beta},\frac{1}{3\beta}\Big\}=\left\{\aligned%
&1 \text{ for } 0<\beta\leqslant \frac{1}{6},\\
&\frac{1+2\beta}{8\beta} \text{ for } \frac{1}{6}\leqslant \beta\leqslant \frac{5}{6},\\
&\frac{1}{3\beta} \text{ for } \frac{5}{6}\leqslant \beta\leqslant
1.
\endaligned\right.
$$
If $F$ is not tangent to $C$ at the point $C\cap Z$, let
$$
\hat{\alpha}\big(S,(1-\beta)C\big)=\mathrm{min}\Big\{1,\frac{1+\beta}{5\beta},\frac{1}{3\beta}\Big\}=\left\{\aligned%
&1 \text{ for } 0<\beta\leqslant \frac{1}{4},\\
&\frac{1+\beta}{5\beta} \text{ for } \frac{1}{4}\leqslant \beta\leqslant \frac{2}{3},\\
&\frac{1}{3\beta} \text{ for } \frac{2}{3}\leqslant \beta\leqslant
1.
\endaligned\right.
$$
In both cases, we have
$\hat{\alpha}(S,(1-\beta)C)\geqslant\check{\alpha}(S,(1-\beta)C)$,
because
$$
\hat{\alpha}\big(S,(1-\beta)C\big)=\mathrm{min}\Big\{\mathrm{lct}\big(S,(1-\beta)C;\beta C\big),\mathrm{lct}\big(S,(1-\beta)C;\beta (2Z+3F)\big)\Big\}.%
$$

\subsection{Blow up of $\mathbb{P}^2$ at two points}%
\label{subsection:deg-7} Suppose that $K_{S}^2=7$. Then there
exists a birational morphism $\pi\colon S\to\mathbb{P}^2$ that is
the blow up of two distinct points in $\mathbb{P}^2$. Denote by
$E_1$ and $E_2$ two $\pi$-exceptional curves, and denote by $L$
the proper transform of the line in $\mathbb{P}^2$ that passes
through $\pi(E_1)$ and $\pi(E_2)$. Then $E_1$, $E_2$, and $L$ are
all $(-1)$-curves in $S$.

The pencil $|E_2+L|$ contains a unique curve that passes though
$C\cap E_1$. Similarly, \mbox{$|E_1+L|$} contains a unique curve
that passes though $C\cap E_2$. Denote these curves by $L_1$ and
$L_2$, respectively. Then $L_1$ is irreducible and smooth unless
$L_1=E_2+L$ (in this case $E_1\cap L\in C$). Similarly, the curve
$L_2$ is irreducible and smooth unless $L_2=E_1+L$ and $L\cap
E_2\in C$.

If $C$ does not contain the points $E_1\cap L$ nor $E_2\cap L$,
then there exists a unique smooth irreducible curve $R\in
|E_1+E_2+L|$ such that $R$ passes though $C\cap L$ and is tangent
to $C$ at the point $C\cap L$. If either $E_1\cap L\in C$ or
$E_2\cap L\in C$, we let $R=E_1+E_2+L$. In the former case,
either $R$ and $C$ have simple tangency at the point $C\cap L$ or
the curve $R$ is tangent to $C$ at the point $C\cap L$ with
multiplicity $3$ (in this case, we must have $R\cap C=C\cap L$, because
$R\cdot C=3$).

If either $E_1\cap L\in C$ or $E_2\cap L\in C$ (but not both,
since $C\cdot L=1$), then we let
$$
\hat{\alpha}\big(S,(1-\beta)C\big)=\mathrm{min}\Big\{1,\frac{1+\beta}{5\beta},\frac{1}{3\beta}\Big\}=\left\{\aligned%
&1 \text{ for } 0<\beta\leqslant \frac{1}{4},\\
&\frac{1+\beta}{5\beta} \text{ for } \frac{1}{4}\leqslant \beta\leqslant \frac{2}{3},\\
&\frac{1}{3\beta} \text{ for } \frac{2}{3}\leqslant \beta\leqslant
1.
\endaligned\right.
$$
If the curve $C$ does not contain the points $E_1\cap L$ nor
$E_2\cap L$, and either $L_1$ is tangent to $C$ at the point
$C\cap E_1$ or $L_2$ is tangent to $C$ at the point $C\cap E_2$,
then we let
$$
\hat{\alpha}\big(S,(1-\beta)C\big)=\mathrm{min}\Big\{1,\frac{1+2\beta}{6\beta},\frac{1}{3\beta}\Big\}=\left\{\aligned%
&1 \text{ for } 0<\beta\leqslant \frac{1}{4},\\
&\frac{1+2\beta}{6\beta} \text{ for } \frac{1}{4}\leqslant \beta\leqslant \frac{1}{2},\\
&\frac{1}{3\beta} \text{ for } \frac{1}{2}\leqslant \beta\leqslant
1.
\endaligned\right.
$$
If the curve $C$ does not contain the points $E_1\cap L$ nor
$E_2\cap L$ (this implies that the curve $R$ is smooth), neither
$L_1$ is tangent to $C$ at the point $C\cap E_1$ nor $L_2$ is
tangent to $C$ at the point $C\cap E_2$, and the curve $R$ is
tangent to $C$ at the point $C\cap L$ with multiplicity $3$, then
we let
$$
\hat{\alpha}\big(S,(1-\beta)C\big)=\mathrm{min}\Big\{1,\frac{1+3\beta}{7\beta},\frac{1}{3\beta}\Big\}=\left\{\aligned%
&1 \text{ for } 0<\beta\leqslant \frac{1}{4},\\
&\frac{1+3\beta}{7\beta} \text{ for } \frac{1}{4}\leqslant \beta\leqslant \frac{4}{9},\\
&\frac{1}{3\beta} \text{ for } \frac{4}{9}\leqslant \beta\leqslant
1.
\endaligned\right.
$$
Finally, if the curve $C$ does not contain the points $E_1\cap L$
nor $E_2\cap L$ (and hence the curve $R$ is smooth), neither $L_1$ is tangent to $C$ at the point
$C\cap E_1$ nor $L_2$ is tangent to $C$ at the point $C\cap E_2$,
and $R$ is tangent to $C$ at the point $C\cap
L$ with multiplicity $2$, then we let
$$
\hat{\alpha}\big(S,(1-\beta)C\big)=\mathrm{min}\Big\{1,\frac{1}{3\beta}\Big\}=\left\{\aligned%
&1 \text{ for } 0<\beta\leqslant \frac{1}{3},\\
&\frac{1}{3\beta} \text{ for } \frac{1}{3}\leqslant \beta\leqslant
1.
\endaligned\right.
$$

We have
$\hat{\alpha}(S,(1-\beta)C)\geqslant\check{\alpha}(S,(1-\beta)C)$.
Indeed, if either $E_1\cap L\in C$ or $E_2\cap L\in C$, then
$$
\hat{\alpha}\big(S,(1-\beta)C\big)=\min\Big\{\mathrm{lct}\big(S,(1-\beta)C;\beta C\big), \mathrm{lct}\big(S,(1-\beta)C;\beta(3L+2E_1+2E_2)\big)\Big\},%
$$
which implies that
$\hat{\alpha}(S,(1-\beta)C)\geqslant\check{\alpha}(S,(1-\beta)C)$.
If neither $E_1\cap L\in C$ nor $E_2\cap L\in C$, then
$$
\min\Big\{\mathrm{lct}\big(S,(1-\beta)C;\beta
C\big),\mathrm{lct}\big(S,(1-\beta)C;\beta(3L+2E_1+2E_2)\big)\Big\}=\min\left\{1,\frac{1}{3\beta}\right\}.
$$
If the curve $C$ does not contain the points $E_1\cap L$ nor
$E_2\cap L$, and $L_1$ is tangent to $C$ at the point
$C\cap E_1$, then
$$
\hat{\alpha}\big(S,(1-\beta)C\big)=\min\Big\{1,
\frac{1}{3\beta},\mathrm{lct}\big(S,(1-\beta)C;\beta(2L_1+2E_1+L)\big)\Big\},
$$
and similarly  if $L_2$ is tangent to $C$ at the point $C\cap E_2$.
If the curve $C$ does not contain the points $E_1\cap L$ nor
$E_2\cap L$ (this implies that the curve $R$ is smooth), neither
$L_1$ is tangent to $C$ at the point $C\cap E_1$ nor $L_2$ is
tangent to $C$ at the point $C\cap E_2$, and the curve $R$ is
tangent to $C$ at the point $C\cap L$ with multiplicity $3$, then
$$
\min\Big\{\mathrm{lct}\big(S,(1-\beta)C;\beta
C\big),\mathrm{lct}\big(S,(1-\beta)C;\beta(3L+2E_1+2E_2)\big),\mathrm{lct}\big(S,(1-\beta)C;\beta(L+2R)\big)\Big\}
$$
equals $\hat{\alpha}(S,(1-\beta)C)$. We conclude that
$\hat{\alpha}(S,(1-\beta)C)\geqslant\check{\alpha}(S,(1-\beta)C)$
in every case.

\subsection{Blow up of $\mathbb{P}^2$ at three points}
\label{subsection:deg-6} Suppose that $K_{S}^2=6$. Then there
exists a birational morphism $\pi\colon S \to \mathbb{P}^2$ that
is the blow up of three non-colinear points. Denote the
$\pi$-exceptional curves by $E_1$, $E_2$, $E_3$, denote the proper
transform on $S$ of the line in $\mathbb{P}^2$ that passes through
$\pi(E_1)$ and $\pi(E_2)$ by $L_{12}$, denote  the proper
transform on $S$ of the line in $\mathbb{P}^2$ that passes through
$\pi(E_1)$ and $\pi(E_3)$ by $L_{13}$, and denote  the proper
transform on $S$ of the line in $\mathbb{P}^2$ that passes through
$\pi(E_2)$ and $\pi(E_3)$ by $L_{23}$. Then $E_1$, $E_2$, $E_3$,
$L_{12}$, $L_{13}$ and $L_{23}$ are all lines in $S$.

If the curve $C$ contains an intersection point of two
intersecting lines in $S$, then we let
$$
\hat{\alpha}\big(S,(1-\beta)C\big)=\mathrm{min}\Big\{1,\frac{1+\beta}{4\beta}\Big\}=\left\{\aligned
&1 \text{ for } 0<\beta\leqslant \frac{1}{3},\\
&\frac{1+\beta}{4\beta} \text{ for } \frac{1}{3}\leqslant
\beta\leqslant 1.
\endaligned\right.
$$
If the curve $C$ does not contain the intersection points of any
two intersecting lines, and there are a line $Z_1$ and an irreducible conic
$Z_2$ in $S$ such that $Z_2$ is tangent to $C$ at the point $C\cap
Z_1$, then we let
$$
\hat{\alpha}\big(S,(1-\beta)C\big)=\mathrm{min}\Big\{1,\frac{1+2\beta}{5\beta},\frac{1}{2\beta}\Big\}=\left\{\aligned
&1 \text{ for } 0<\beta\leqslant \frac{1}{3},\\
&\frac{1+2\beta}{5\beta}  \text{ for } \frac{1}{3}\leqslant \beta\leqslant \frac{3}{4},\\
&\frac{1}{2\beta} \text{ for } \frac{3}{4}\leqslant \beta\leqslant
1.
\endaligned\right.
$$
If $C$ does not contain the intersection point of any two
intersecting lines, and for every line $Z_1$ in $S$, there exists
no irreducible conic $Z_2$ in $S$ such that $Z_2$ is tangent to
$C$ at  $C\cap Z_1$, then we~let
$$
\hat{\alpha}\big(S,(1-\beta)C\big)=\mathrm{min}\Big\{1,\frac{1}{2\beta}\Big\}=\left\{\aligned
&1 \text{ for } 0<\beta\leqslant \frac{1}{2},\\
&\frac{1}{2\beta} \text{ for } \frac{1}{2}\leqslant \beta\leqslant
1.
\endaligned\right.
$$

One has
$\hat{\alpha}(S,(1-\beta)C)\geqslant\check{\alpha}(S,(1-\beta)C)$.
Indeed, we have $2E_1+2L_{12}+L_{13}+E_2\sim-K_S$. Thus, if
$E_1\cap L_{12}\not\in C$, $E_1\cap L_{13}\not\in C$ and $E_2\cap
L_{12}\not\in C$, then
$$
\min\Big\{\mathrm{lct}\big(S,(1-\beta)C;\beta
C\big),\mathrm{lct}\big(S,(1-\beta)C;\beta(2E_1+2L_{12}+L_{13}+E_2)\big)\Big\}=\left\{\aligned
&1 \text{ for } 0<\beta\leqslant \frac{1}{2},\\
&\frac{1}{2\beta} \text{ for } \frac{1}{2}\leqslant \beta\leqslant
1.
\endaligned\right.
$$
Otherwise, this minimum is $\hat{\alpha}(S,(1-\beta)C)$. This
shows that
$\hat{\alpha}(S,(1-\beta)C)\geqslant\check{\alpha}(S,(1-\beta)C)$
except for the case when $C$ does not contain the
intersection point of any two intersecting lines, but there are a
line $Z_1$ and a conic $Z_2$ in $S$ such that $Z_2$ is tangent to
$C$ at the point $C\cap Z_1$. In the latter case, we may assume
that $Z_1=E_1$ and $Z_2\in |L_{12}+E_2|$, which implies that
$$
\hat{\alpha}\big(S,(1-\beta)C\big)=\min\Big\{\mathrm{lct}\big(S,(1-\beta)C;\beta
C\big),\mathrm{lct}\big(S,(1-\beta)C;\beta(2Z_2+E_1+L_{23})\big)\Big\},
$$
since $2Z_2+E_1+L_{23}\sim -K_{S}$. Thus, in all cases we have
$\hat{\alpha}(S,(1-\beta)C)\geqslant\check{\alpha}(S,(1-\beta)C)$.

\subsection{Blow up of
$\mathbb{P}^2$ at four points} \label{subsection:deg-5} Suppose
that $K_{S}^2=5$. Then there exists a birational morphism
$\pi\colon S\to\mathbb{P}^2$ that contracts four smooth rational
curves to four points such that no three of them are colinear.
Denote these curves by $E_1$, $E_2$, $E_3$, $E_4$. For and
integers $i$ and $j$ such that $1\leqslant i<j\leqslant 4$, denote by $L_{ij}$
the proper transform on $S$ via $\pi$ of the line in
$\mathbb{P}^2$ that passes through $\pi(E_i)$ and $\pi(E_j)$.
These gives us six lines $L_{12}$, $L_{13}$, $L_{14}$,
$L_{23}$, $L_{24}$ and $L_{34}$. Moreover, $E_1$, $E_1$, $E_2$,
$E_3$, $E_4$, $L_{12}$, $L_{13}$, $L_{14}$, $L_{23}$, $L_{24}$ and
$L_{34}$ are all lines in $S$. Let
$$
\hat{\alpha}\big(S,(1-\beta)C\big)=\mathrm{min}\Big\{1,\frac{1}{2\beta}\Big\}=\left\{\aligned
&1 \text{ for } 0<\beta\leqslant \frac{1}{2},\\
&\frac{1}{2\beta} \text{ for } \frac{1}{2}\leqslant \beta\leqslant 1.\\
\endaligned\right.
$$
Then
$\hat{\alpha}(S,(1-\beta)C)\geqslant\check{\alpha}(S,(1-\beta)C)$,
since $2E_1+L_{12}+L_{13}+L_{14}\sim-K_S$ and
$$
\hat{\alpha}\big(S,(1-\beta)C\big)=\min\Big\{\mathrm{lct}\big(S,(1-\beta)C;\beta C\big), \mathrm{lct}\big(S,(1-\beta)C;\beta(2E_1+L_{12}+L_{13}+L_{14}\big)\big)\Big\}.%
$$

\subsection{Complete intersections of two quadrics}
\label{subsection:deg-4} Suppose that $K_{S}^2=4$. Then there
exists a birational morphism $\pi\colon S\to\mathbb{P}^2$ that is
the blow up of five points such that no three of them are colinear.
Denote by $E_1$, $E_2$, $E_3$, $E_4$ and $E_5$ the
$\pi$-exceptional curves. For any integers $i$ and $j$ such that
$1\leqslant i<j\leqslant 5$, denote by $L_{ij}$ the proper transform via
$\pi$ on $S$ of the line in $\mathbb{P}^2$ that passes through
$\pi(E_i)$ and $\pi(E_j)$. Denote by $E$ the proper transform
on $S$ of the unique smooth conic in $\mathbb{P}^2$ that passes
through $\pi(E_1)$, $\pi(E_2)$, $\pi(E_3)$, $\pi(E_4)$ and
$\pi(E_5)$. Then $E_1$, $E_2$, $E_3$, $E_4$, $E_5$,
$L_{12}$, $L_{13}$, $L_{14}$, $L_{15}$, $L_{23}$, $L_{24}$,
$L_{25}$, $L_{34}$, $L_{35}$, $L_{45}$ and $E$ are all the lines in
$S$.

If the curve $C$ contains the intersection point of any two
intersecting lines, then we let
$$
\hat{\alpha}\big(S,(1-\beta)C\big)=\mathrm{min}\Big\{1,\frac{1+\beta}{3\beta}\Big\}=\left\{\aligned
&1 \text{ for } 0<\beta\leqslant \frac{1}{2},\\
&\frac{1+\beta}{3\beta} \text{ for } \frac{1}{2}\leqslant \beta\leqslant 1.\\
\endaligned\right.
$$
If the curve $C$ does not contain the intersection point of any
two intersecting lines, but there are two conics $C_1$ and $C_2$
in $S$ such that $C_1+C_2\sim -K_{S}$, and $C_1$ and $C_2$ both
tangent $C$ at one point, then we let
$$
\hat{\alpha}\big(S,(1-\beta)C\big)=\mathrm{min}\Big\{1,\frac{1+2\beta}{4\beta},\frac{2}{3\beta}\Big\}=\left\{\aligned
&1 \text{ for } 0<\beta\leqslant \frac{1}{2},\\
&\frac{1+2\beta}{4\beta}  \text{ for } \frac{1}{2}\leqslant \beta\leqslant \frac{5}{6},\\
&\frac{2}{3\beta} \text{ for } \frac{5}{6}\leqslant \beta\leqslant
1.
\endaligned\right.
$$
Finally, if the curve $C$ does not contain the intersection point
of any two intersecting lines, and for every two conics $C_1$ and
$C_2$ in $S$ such that $C_1+C_2\sim -K_{S}$, the conics $C_1$ and
$C_2$ do not tangent $C$ at one point, then we let
$$
\hat{\alpha}\big(S,(1-\beta)C\big)=\mathrm{min}\Big\{1,\frac{2}{3\beta}\Big\}=\left\{\aligned
&1 \text{ for } 0<\beta\leqslant \frac{2}{3},\\
&\frac{2}{3\beta} \text{ for } \frac{2}{3}\leqslant \beta\leqslant
1.
\endaligned\right.
$$

We claim that
$\hat{\alpha}(S,(1-\beta)C)\geqslant\check{\alpha}(S,(1-\beta)C)$.
Indeed, the lines $L_{12}$ and $L_{34}$ intersect at a single
point. Let $Z$ be the proper transform on $S$ of the line in
$\mathbb{P}^2$ that passes through $\pi(E_5)$ and $\pi(L_{12}\cap
L_{34})$. Then $L_{12}+L_{34}+Z\sim -K_{S}$. Moreover,  if
$L_{12}\cap L_{34}\in C$, then
$$
\min\Big\{\mathrm{lct}\big(S,(1-\beta)C;\beta
C\big),\mathrm{lct}\big(S,(1-\beta)C;\beta(L_{12}+L_{34}+Z)\big)\Big\}=\left\{\aligned
&1 \text{ for } 0<\beta\leqslant \frac{1}{2},\\
&\frac{1+\beta}{3\beta} \text{ for } \frac{1}{2}\leqslant
\beta\leqslant 1.
\endaligned\right.
$$
However, if $L_{12}\cap L_{34}\not\in C$, then this minimum equals
$\min\{1,\frac{2}{3\beta}\}$. Since we can repeat these
computations for any pair of intersecting lines in $S$, we see
that
$\hat{\alpha}(S,(1-\beta)C)\geqslant\check{\alpha}(S,(1-\beta)C)$
except possibly the case when $C$ does not contain the
intersection point of any two intersecting lines, but there are
two conics $C_1$ and $C_2$ in $S$ such that $C_1+C_2\sim -K_{S}$,
and $C_1$ and $C_2$ both tangent $C$ at one point. In the latter
case, $\hat{\alpha}(S,(1-\beta)C)$ is equal to
$$
\min\Big\{\mathrm{lct}\big(S,(1-\beta)C;\beta
C\big),\mathrm{lct}\big(S,(1-\beta)C;\beta(L_{12}+L_{34}+Z)\big),\mathrm{lct}\big(S,(1-\beta)C;\beta(C_1+C_2)\big)\Big\},
$$
since $C_1+C_2\sim -K_{S}$. This shows that
$\hat{\alpha}(S,(1-\beta)C)\geqslant\check{\alpha}(S,(1-\beta)C)$
in all three cases.

\subsection{Cubic surfaces}
\label{subsection:deg-3} Suppose that $K_{S}^2=3$. Then $S$ is a
smooth cubic surface in $\mathbb{P}^3$. Recall that an Eckardt
point in $S$ is a point of intersection of three lines contained
in $S$. General cubic surface contains no Eckardt points. If $S$
contains an Eckardt point that is contained in  $C$, then we let
$$
\hat{\alpha}\big(S,(1-\beta)C\big)=\mathrm{min}\Big\{1,\frac{1+\beta}{3\beta}\Big\}=\left\{\aligned
&1 \text{ for } 0<\beta\leqslant \frac{1}{2},\\
&\frac{1+\beta}{3\beta} \text{ for } \frac{1}{2}\leqslant \beta\leqslant 1.\\
\endaligned\right.
$$
If $S$ contains an Eckardt point and $C$ contains no Eckardt
points, then we let
$$
\hat{\alpha}\big(S,(1-\beta)C\big)=\mathrm{min}\Big\{1,\frac{2}{3\beta}\Big\}=\left\{\aligned
&1 \text{ for } 0<\beta\leqslant \frac{2}{3},\\
&\frac{2}{3\beta} \text{ for } \frac{2}{3}\leqslant \beta\leqslant 1.\\
\endaligned\right.
$$
If $S$ contains no Eckardt points, but $S$ contains a line $L$ and
a conic $M$ such that $L$ is tangent to $M$ and $L\cap M\in C$,
then we let
$$
\hat{\alpha}\big(S,(1-\beta)C\big)=\mathrm{min}\Big\{1,\frac{2+\beta}{4\beta}\Big\}=\left\{\aligned
&1 \text{ for } 0<\beta\leqslant \frac{2}{3},\\
&\frac{2+\beta}{4\beta} \text{ for } \frac{2}{3}\leqslant \beta\leqslant 1.\\
\endaligned\right.
$$
If $S$ contains no Eckardt points, for every line $L$ and every
conic $M$ on $S$ such that $L$ is tangent to $M$, we have $L\cap
M\not\in C$, but there is a cuspidal curve $T\in|-K_S|$ such that
$T\cap C=\mathrm{Sing}(T)$, then we let
$$
\hat{\alpha}\big(S,(1-\beta)C\big)=\mathrm{min}\Big\{1,\frac{2+3\beta}{6\beta},\frac{3}{4\beta}\Big\}=\left\{\aligned
&1 \text{ for } 0<\beta\leqslant \frac{2}{3},\\
&\frac{2+3\beta}{6\beta} \text{ for } \frac{2}{3}\leqslant \beta\leqslant \frac{5}{6},\\
&\frac{3}{4\beta} \text{ for } \frac{5}{6}\leqslant \beta\leqslant 1.\\
\endaligned\right.
$$
Finally, if $S$ contains no Eckardt points, for every line $L$ and
every conic $M$ on $S$ such that $L$ is tangent to $M$ we have
$L\cap M\not\in C$, and every irreducible cuspidal curve
$T\in|-K_S|$ intersects $C$ by at least two point, then we let
$$
\hat{\alpha}\big(S,(1-\beta)C\big)=\mathrm{min}\Big\{1,\frac{3}{4\beta}\Big\}=\left\{\aligned
&1 \text{ for } 0<\beta\leqslant \frac{3}{4},\\
&\frac{3}{4\beta} \text{ for } \frac{3}{4}\leqslant \beta\leqslant 1.\\
\endaligned\right.
$$
One can easily check that
$\hat{\alpha}(S,(1-\beta)C)\geqslant\check{\alpha}(S,(1-\beta)C)$
(see \cite[Theorem 4.9.1]{JMG-thesis}).

\subsection{Double covers of $\mathbb P^2$}
\label{subsection:deg-2}

Suppose that $K_{S}^2=2$. If $|-K_S|$ contains a tacnodal curve
whose singular point is contained in $C$, then we let
$$
\hat{\alpha}\big(S,(1-\beta)C\big)=\mathrm{min}\Big\{1,\frac{2+\beta}{4\beta}\Big\}=\left\{\aligned
&1 \text{ for } 0<\beta\leqslant \frac{2}{3},\\
&\frac{2+\beta}{4\beta} \text{ for } \frac{2}{3}\leqslant \beta\leqslant 1.\\
\endaligned\right.
$$
If $|-K_S|$ contains a tacknodal curve, but $C$ does not contain
singular points of all tacknodal curves in $|-K_S|$, then we let
$$
\hat{\alpha}\big(S,(1-\beta)C\big)=\mathrm{min}\Big\{1,\frac{3}{4\beta}\Big\}=\left\{\aligned
&1 \text{ for } 0<\beta\leqslant \frac{3}{4},\\
&\frac{3}{4\beta} \text{ for } \frac{3}{4}\leqslant \beta\leqslant 1.\\
\endaligned\right.
$$
If $|-K_S|$ contains no curves with tacnodal singularities, but
$C$ contains the cuspidal singular point of a cuspidal rational curve in $|-K_S|$, then we let
$$
\hat{\alpha}\big(S,(1-\beta)C\big)=\mathrm{min}\Big\{1,\frac{3+2\beta}{6\beta}\Big\}=\left\{\aligned
&1 \text{ for } 0<\beta\leqslant \frac{3}{4},\\
&\frac{3+2\beta}{6\beta} \text{ for } \frac{3}{4}\leqslant \beta\leqslant 1.%
\endaligned\right.
$$
Finally, if $|-K_S|$ contains no curves with tacnodal
singularities, and $C$ does not contain cuspidal singular points
of all cuspidal rational curves in $|-K_S|$, then we let
$$
\hat{\alpha}\big(S,(1-\beta)C\big)=\mathrm{min}\Big\{1,\frac{5}{6\beta}\Big\}=\left\{\aligned
&1 \text{ for } 0<\beta\leqslant \frac{5}{6},\\
&\frac{5}{6\beta} \text{ for } \frac{5}{6}\leqslant \beta\leqslant 1.\\
\endaligned\right.
$$
One can easily check that
$\hat{\alpha}(S,(1-\beta)C)\geqslant\check{\alpha}(S,(1-\beta)C)$
(see \cite[Theorem 4.10.1]{JMG-thesis}).

\subsection{Double covers of quadric cones}
\label{subsection:deg-1} Suppose that  $K_{S}^2=1$.  If $|-K_{S}|$
contains no cuspidal curves, then we let
$\hat{\alpha}(S,(1-\beta)C)=1$ for every $\beta\in(0,1]$.
Otherwise, we let
$$
\hat{\alpha}\big(S,(1-\beta)C\big)=\mathrm{min}\Big\{1,\frac{5}{6\beta}\Big\}=\left\{\aligned
&1 \text{ for } 0<\beta\leqslant \frac{5}{6},\\
&\frac{5}{6\beta} \text{ for } \frac{5}{6}\leqslant \beta\leqslant 1.\\
\endaligned\right.
$$
In the former case, we have
$\hat{\alpha}(S,(1-\beta)C)=\mathrm{lct}(S,(1-\beta)C;\beta C)$.
In the latter case, we have
$$
\hat{\alpha}\big(S,(1-\beta)C\big)=\min\Big\{\mathrm{lct}\big(S,(1-\beta)C;\beta
C\big),\mathrm{lct}\big(S,(1-\beta)C;\beta Z\big)\Big\},
$$
where $Z$ is a cuspidal curve in $|-K_S|$. Thus,
$\hat{\alpha}(S,(1-\beta)C)\geqslant\check{\alpha}(S,(1-\beta)C)$
in both cases.

\section{Local inequalities}
\label{section:preliminaries} Let $S$ be a~smooth surface, let $D$
be an effective $\mathbb{R}$-divisor on $S$, and let $P$ be a
point in $S$.

\begin{lemma}
\label{lemma:Skoda} Suppose that $(S,D)$ is not log canonical at
$P$. Then $\mathrm{mult}_{P}(D)>1$.
\end{lemma}

\begin{proof}
This is a well-known fact. See \cite[Exercise~6.18]{CoKoSm03}, for
instance.
\end{proof}

\begin{lemma}
\label{lemma:convexity} Suppose that $(S,D)$ is not log canonical
at $P$. Let $B$ be an effective $\mathbb{R}$-divisor on $S$ such
that $(S,B)$ is log canonical and $B\sim_{\mathbb{R}} D$. Then
there exists an effective $\mathbb{R}$-divisor $D^\prime$ on $S$
such that $D^\prime\sim_{\mathbb{R}} D$, the log pair
$(S,D^\prime)$ is not log canonical at $P$, and
$\mathrm{Supp}(D^\prime)$ does not contain at least one
irreducible component of $\mathrm{Supp}(B)$.
\end{lemma}

\begin{proof}
Let $\mu$ be the greatest real number such that
$D^\prime:=(1+\mu)D-\mu B$ is effective. Since $D\ne B$, the
number $\mu$ does exist. Then $D^\prime\sim_{\mathbb{R}} D$, the
log pair $(S,D^\prime)$ is not log canonical at $P$, and
$\mathrm{Supp}(D^\prime)$ does not contain at least one
irreducible component of $\mathrm{Supp}(B)$.
\end{proof}

Let $\pi_1\colon S_1\to S$ be a~blow up of the point $P$, let
$F_1$ be the $\pi$-exceptional curve, and let $D^1$ be the~proper
transform of $D$ via $\pi_1$. Then
$K_{S_1}+D^1+(\mathrm{mult}_{P}(D)-1)F_1\sim_{\mathbb{R}}\pi_1^{*}(K_{S}+D)$.

\begin{lemma}
\label{lemma:log-pull-back} Suppose that $(S,D)$ is not log
canonical at $P$.  Then $\mathrm{mult}_{P}(D)>1$ and there exists
a point $P_1\in F_1$ such that $(S_1,
D^1+(\mathrm{mult}_{P}(D)-1)F_1)$ is not log canonical at $P_1$.
Moreover, one has
$\mathrm{mult}_{P}(D)+\mathrm{mult}_{P_1}(D^1)>2$. If, in
addition, $\mathrm{mult}_{P}(D)\leqslant 2$, then such point $P_1$
is unique.
\end{lemma}

\begin{proof}
This is a well-known fact. See, for example,
\cite[Remark~2.5]{ChePark13}.
\end{proof}

Let $C$ be an irreducible curve on $S$ that contains $P$. Suppose
that $C$ is smooth at $P$. Write $D=aC+\Omega$, where
$a\in\mathbb{R}_{\geqslant 0}$, and $\Omega$ is an effective
$\mathbb{R}$-divisor on $S$ with
$C\not\subset\mathrm{Supp}(\Omega)$.

\begin{theorem}
\label{theorem:adjunction} If $(S,aC+\Omega)$ is not log canonical
at $P$ and  $a\leqslant 1$, then $\mathrm{mult}_{P}(\Omega\cdot
C)>1$.
\end{theorem}

\begin{proof}
See, for example,
\cite[Exercise~6.31]{CoKoSm03}, \cite[Lemma~2.5]{JMG2014} or \cite[Theorem~7]{Che13}.
\end{proof}

Denote the proper transform of the curve $C$ on the surface $S_1$
 by $C^1$, and denote the proper transform of the
$\mathbb{R}$-divisor $\Omega$ on the surface $S_1$ by $\Omega^1$.

\begin{lemma}
\label{lemma:adjunction-small-mult} Suppose that $a\leqslant 1$,
the log pair $(S,aC+\Omega)$ is not log canonical at the point
$P$, and $\mathrm{mult}_{P}(\Omega)\leqslant 1$. Then $(S_1,
aC^1+\Omega^1+(a+\mathrm{mult}_{P}(\Omega)-1)F_1)$ is not log
canonical at $C^1\cap F_1$, it is log canonical at every point in
$E_1\setminus (C^1\cap F_1)$, and $\mathrm{mult}_{P}(\Omega\cdot
C)>2-a$.
\end{lemma}

\begin{proof}
Since $a\leqslant 1$ and $\mathrm{mult}_{P}(\Omega)\leqslant 1$,
we have $\mathrm{mult}_{P}(D)\leqslant 2$. By
Lemma~\ref{lemma:log-pull-back}, there exists a unique point
$P_1\in F_1$ such that the log pair $(S_1,
aC^1+\Omega^1+(a+\mathrm{mult}_{P}(\Omega)-1)F_1)$ is not log
canonical at $P_1$. If $P_1\not\in C^1$, then
$\mathrm{mult}_{P}(\Omega)=F_1\cdot\Omega^1\geqslant\mathrm{mult}_{P_1}(\Omega^1\cdot
F_1)>1$ by Theorem~\ref{theorem:adjunction}, which is impossible,
since $\mathrm{mult}_P(\Omega)\leqslant 1$. Thus, $P_1\in C^1$.
Then, by Theorem~\ref{theorem:adjunction} again:
\[
\mathrm{mult}_{P}\Big(\Omega\cdot C\Big)\geqslant\mathrm{mult}_{P}(\Omega)+\mathrm{mult}_{P_1}\Big(\Omega^1\cdot C^1\Big)>2-a.\qedhere%
\]
\end{proof}

Let us consider an \emph{infinite} sequence of blow ups
$$
\xymatrix{\cdots\ar@{->}[rr]^{\pi_{n+1}}&&S_n\ar@{->}[rr]^{\pi_n}&&
S_{n-1}\ar@{->}[rr]^{\pi_{n-1}}&& \cdots\ar@{->}[rr]^{\pi_3}&&
S_2\ar@{->}[rr]^{\pi_2}&& S_1\ar@{->}[rr]^{\pi_1}&&S}
$$
such that each $\pi_n$ is the blow up of the point in the proper
transform of the curve $C$ on the surface $S_{n-1}$ that dominates
$P$. Denote the $\pi_{n}$-exceptional curve by $F_n$, and denote
the proper transform of  $C$ on   $S_{n}$ by $C^n$. For every
$n\geqslant 1$, write $P_n=C^n\cap F_n$, denote the proper
transform of the divisor $\Omega$ on  $S_n$ by $\Omega^n$, let
$m_n=\mathrm{mult}_{P_n}(\Omega^n)$ and let
$m_0=\mathrm{mult}_{P}(\Omega)$. For every positive integers
$k\leqslant n$, denote the proper transform of the curve $F_k$ on
 $S_n$  by $F_k^n$. Finally, we let
$$
D^{S_n}=aC^n+\Omega^n+\sum_{k=1}^{n}\Big(k
a-k+\sum_{i=0}^{k-1}m_i\Big)F_k^n
$$
for every $n\geqslant 1$. Then $K_{S_n}+D^{S_n}\sim_{\mathbb{R}}
(\pi_1\circ\pi_2\circ\cdots\circ\pi_n)^{*}(K_{S}+D)$ for every
$n\geqslant 1$.

\begin{theorem}
\label{theorem:blow-ups} Suppose that $(S,aC+\Omega)$ is not log
canonical at $P$ and $a\leqslant 1$. Then $m_0+a>1$ and
$\mathrm{mult}_{P}(\Omega\cdot C)>1$. Moreover, the following
additional assertions hold:
\begin{enumerate}[(i)]
\item if $m_0\leqslant 1$, then the log pair $(S_1, D^{S_1})$ is
not log canonical at $P_1$,

\item if $(S_n, D^{S_n})$ is not log canonical at some point in
$F_n$, then $D^{S_n}$ is an effective divisor,

\item if $(S_n, D^{S_n})$ is not log canonical at some point in
$F_n$ and $\sum_{i=0}^{n-1}m_i\leqslant n+1-na$, then such point
in $F_n$ is unique,

\item if $(S_{n}, D^{S_{n}})$ is not log canonical at $P_{n}$,
then $(n+1)a+\sum_{i=0}^{n}m_i>n+2$, the log pair $(S_{n+1},
D^{S_{n+1}})$ is not log canonical at some point in $F_{n+1}$, and
$\mathrm{mult}_{P}(\Omega\cdot C)>n+1-na$,

\item if $n\geqslant 2$, $m_{n-1}\leqslant 1$ and
$\sum_{i=0}^{n-1}m_i\leqslant n+1-na$, then $(S_n, D^{S_n})$ is
log canonical at every point of $F_n$ different from $P_n$ and
$F_n\cap F_{n-1}^n$,

\item if $n\geqslant 2$ and $\sum_{i=0}^{n-1}m_i\leqslant
n-(n-1)a$, then $(S_{n}, D^{S_{n}})$ is log canonical at $F_n\cap
F_{n-1}^n$,

\item if $n\geqslant 2$, $\sum_{i=0}^{n-2}m_i\leqslant n-(n-1)a$,
and $\sum_{i=0}^{n-3}m_i+2m_{n-2}\leqslant n+1-na$, then $(S_{n},
D^{S_{n}})$ is log canonical at $F_n\cap F_{n-1}^n$.
\end{enumerate}
\end{theorem}

\begin{proof}
By Lemma~\ref{lemma:Skoda}, we have $m_0+a>1$. By
Theorem~\ref{theorem:adjunction}, we have
$\mathrm{mult}_{P}(\Omega\cdot C)>1-a$. Assertion (i) follows from
Lemma~\ref{lemma:adjunction-small-mult}. If $(S_n, D^{S_n})$ is
not log canonical at some point in $F_n$, then
$(S_{n-1},D^{S_{n-1}})$ is not log canonical at $P^{n-1}$. Thus,
assertion (ii) follows from Lemma~\ref{lemma:Skoda}. Inequality
$\sum_{i=0}^{n-1}m_i\leqslant n+1-na$ is equivalent to
$\mathrm{mult}_{P_{n-1}}(D^{S_{n-1}})\leqslant 2$. Thus, assertion
(iii) follows from Lemma~\ref{lemma:log-pull-back}. If $(S_{n},
D^{S_{n}})$ is not log canonical at $P_{n}$, then
\mbox{$(n+1)a+\sum_{i=0}^{n}m_i>n+2$} by Lemma~\ref{lemma:Skoda},
the pair $(S_{n+1}, D^{S_{n+1}})$ is not log canonical at some
point in $F_{n+1}$ by Lemma~\ref{lemma:log-pull-back},~and
$$
\mathrm{mult}_{P}\Big(\Omega\cdot
C\Big)-\sum_{i=0}^{n-1}m_i=\mathrm{mult}_{P_n}\Big(\Omega^n\cdot
C^n\Big)>1-\Big(na-n+\sum_{i=0}^{n-1}m_i\Big),
$$
by Theorem~\ref{theorem:adjunction}. This proves assertion (iv).

Suppose that $n\geqslant 2$. Let $O=F_n\cap F_{n-1}^n$.  If
$\sum_{i=0}^{n-1}m_i\leqslant n+1-na$ and $(S_n, D^{S_n})$ is not
log canonical at some point in $F_n\setminus (P_n\cup O)$, then
$m_{n-1}=F_n\cdot\Omega^n>1$ by Theorem~\ref{theorem:adjunction},
which implies assertion~(v). If $(S_{n}, D^{S_{n}})$ is not log
canonical at $O$ and $\sum_{i=0}^{n-1}m_i\leqslant n+1-na$, then
$$
m_{n-1}=F_{n}\cdot\Omega^n\geqslant
\mathrm{mult}_{O}\Big(F_{n}\cdot\Omega^n\Big)>1-\Big((n-1)a-n+1+\sum_{i=0}^{n-2}m_i\Big)
$$
by Theorem~\ref{theorem:adjunction}. If $(S_{n}, D^{S_{n}})$ is
not log canonical at $O$ and $\sum_{i=0}^{n-2}m_i\leqslant
n-(n-1)a$, then
$$
m_{n-2}-m_{n-1}=F_{n-1}^n\cdot\Omega^n\geqslant
\mathrm{mult}_{O}\Big(F_{n-1}^n\cdot\Omega^n\Big)>1-\Big(na-n+\sum_{i=0}^{n-1}m_i\Big)
$$
by Theorem~\ref{theorem:adjunction}. This proves assertions (vi)
and (vii).
\end{proof}

\begin{corollary}
\label{corollary:2012-summer} Suppose that $(S,aC+\Omega)$ is not
log canonical at $P$, $C\not\subset\mathrm{Supp}(\Omega)$,
$a\leqslant 1$ and $m_0\leqslant \min\{1,1+\frac{1}{n}-na\}$ for
some integer $n\geqslant 1$. Then $\mathrm{mult}_{P}(\Omega\cdot
C)>n+1-na$.
\end{corollary}

\begin{corollary}
\label{corollary:4-blow-ups} Suppose that $(S,aC+\Omega)$ is not
log canonical at $P$, $a\leqslant 1$ and $m_0\leqslant 1$. Suppose
that $2m_0\leqslant 3-2a$ or $m_0+m_1\leqslant 2-a$. Suppose that
$m_0+2m_1\leqslant 4-3a$ or \mbox{$m_0+m_1+m_2\leqslant 3-2a$}.
Then $\mathrm{mult}_{P}(\Omega\cdot C)>4-3a$. If
\mbox{$m_0+m_1+2m_2\leqslant 5-4a$} or
\mbox{$m_0+m_1+m_2+m_3\leqslant 4-3a$}, then
$\mathrm{mult}_{P}(\Omega\cdot C)>5-4a$.
\end{corollary}

Let us conclude this section by recalling

\begin{theorem}[{\cite[Theorem~13]{Che13}}]
\label{theorem:Trento} Let $C_1$ and $C_2$ be two irreducible
curves on $S$ that are both smooth at $P$ and intersect
transversally at $P$. Let $D=a_1C_1+a_2C_2+\Delta$, where $a_1$
and $a_2$ are non-negative real numbers, and $\Delta$ is an
effective $\mathbb{R}$-divisor on  $S$ whose support does not
contain the curves $C_1$ and $C_2$. If $(S,D)$ is not log
canonical at $P$ and $\mathrm{mult}_{P}(\Delta)\leqslant 1$, then
$\mathrm{mult}_{P}(\Delta\cdot C_{1})>2(1-a_{2})$ or
$\mathrm{mult}_{P}(\Delta\cdot C_{2})>2(1-a_{1})$.
\end{theorem}

\section{The proof}
\label{section:proof}

Let us use the notation of Section~\ref{section:formulas}. The goal of
this section is to prove

\begin{theorem}
\label{theorem:main-strong} One has
$\alpha(S,(1-\beta)C)=\hat{\alpha}(S,(1-\beta)C)$ for every
$\beta\in(0,1]$.
\end{theorem}

This theorem implies Theorem~\ref{theorem:main}, since
$\hat{\alpha}(S, (1-\beta)C)\geqslant\check{\alpha}(S,
(1-\beta)C)$ (see Section~\ref{section:formulas}) and
$\check{\alpha}(S, (1-\beta)C)\geqslant\alpha(S, (1-\beta)C)$ (by
definition) for every $\beta\in(0,1]$.

Let $D$ be \emph{any} effective $\mathbb{R}$-divisor such that
$D\sim_{\mathbb{R}}-K_S$, and let $P$ be \emph{any} point in $S$.
Since $\alpha(S,(1-\beta)C)\leqslant\hat{\alpha}(S,(1-\beta)C)$,
to prove Theorem~\ref{theorem:main-strong}, it is enough to show
that the log pair
\begin{equation}
\label{equation:log-pair}
\Big(S,(1-\beta)C+\hat{\alpha}(S,(1-\beta)C)\beta D\Big)%
\end{equation}
is log canonical at $P$ for every $\beta\in(0,1]$. We will do this
in several steps.

\begin{lemma}
\label{lemma:P-in-C} Suppose that \eqref{equation:log-pair} is not
log canonical at $P$. Then $P\in C$, we have
$$
\mathrm{mult}_{P}(D)>\frac{1}{\hat{\alpha}(S,(1-\beta)C)}\geqslant 1,%
$$
and \eqref{equation:log-pair} is log canonical outside of the
point $P$. Moreover, if there exists a $(-1)$-curve $Z\subset S$
such that $P\in Z$, then $Z\subset\mathrm{Supp}(D)$. Furthermore,
there exists an effective $\mathbb{R}$-divisor
$D^\prime\sim_{\mathbb R} D$ such that
$C\not\subset\mathrm{Supp}(D^\prime)$ and
$(S,(1-\beta)C+\hat{\alpha}(S,(1-\beta)C)\beta D^\prime)$ is not
log canonical at $P$.
\end{lemma}

\begin{proof}
If $P\not\in C$, then $(S,\hat{\alpha}(S,(1-\beta)C)\beta D)$ is
not log canonical at $P$, which is impossible, since
$\alpha(S)\leqslant\beta\hat{\alpha}(S,(1-\beta)C)$ by
\cite[Theorem~1.7]{Ch07b}. We have
$\hat{\alpha}(S,(1-\beta)C)\mathrm{mult}_{P}(D)>1$ by
Lemma~\ref{lemma:Skoda}. In particular, if there exists a
$(-1)$-curve $Z\subset S$ such that $P\in Z$, then $Z$ must be
contained in $\mathrm{Supp}(D)$, because otherwise we would have
$1=Z\cdot D\geqslant\mathrm{mult}_{P}(D)>1$.

We see that \eqref{equation:log-pair} is log canonical outside of
the curve $C$. Moreover, the coefficient of the curve $C$ in the
divisor $(1-\beta)C+\hat{\alpha}(S,(1-\beta)C)\beta D$ does not
exceed $1$, since $D\sim_{\mathbb{R}} C$. Hence, the log pair
\eqref{equation:log-pair} is log canonical outside of finitely
many points. Now the connectedness principle (see, for example,
\cite[Theorem~6.32]{CoKoSm03}) implies that
\eqref{equation:log-pair} is log canonical outside of $P$.

Since $(S,(1-\beta)C+\hat{\alpha}(S,(1-\beta)C)\beta C)$ is log
canonical, it follows from Lemma~\ref{lemma:convexity} that there
is an effective $\mathbb{R}$-divisor $D^\prime\sim_{\mathbb R} D$
such that $C\not\subset\mathrm{Supp}(D^\prime)$ and
$(S,(1-\beta)C+\hat{\alpha}(S,(1-\beta)C)\beta D^\prime)$ is not
log canonical at $P$.
\end{proof}

Thus, to prove that \eqref{equation:log-pair} is log canonical at
$P$, we may assume that $P\in C\not\subset\mathrm{Supp}(D)$.

\begin{lemma}
\label{lemma:plane} If $S\cong\mathbb{P}^2$, then
\eqref{equation:log-pair} is log canonical at $P$.
\end{lemma}

\begin{proof}
Suppose \eqref{equation:log-pair} is not log canonical at $P$. Let
$L$ be a general line in $S$ that contains $P$. Then
$\mathrm{mult}_{P}(D)\leqslant D\cdot L=3$. But
$3\hat{\alpha}(S,(1-\beta)C)\beta\leqslant\frac{1}{3}+\beta$ (see
\S\ref{subsection:P2}). Thus, if $\beta\leqslant \frac{2}{3}$,
then
$$
\hat{\alpha}(S,(1-\beta)C)\beta\mathrm{mult}_{P}(D)\leqslant 3\hat{\alpha}(S,(1-\beta)C)\beta\leqslant\frac{1}{3}+\beta\leqslant 1.%
$$
Similarly, if $\frac{2}{3}\leqslant\beta\leqslant 1$, then
$\hat{\alpha}(S,(1-\beta)C)\beta\mathrm{mult}_P(D)\leqslant
\frac{1}{3}\mathrm{mult}_P(D)\leqslant 1$. Applying
Corollary~\ref{corollary:2012-summer} with $n=3$ to
\eqref{equation:log-pair}, we get
$$
9\beta\hat{\alpha}(S,(1-\beta)C)=\hat{\alpha}(S,(1-\beta)C)\beta(C\cdot D)\geqslant\hat{\alpha}(S,(1-\beta)C)\beta\mathrm{mult}_{P}\Big(C\cdot D\Big)>1+3\beta,%
$$
which contradicts the definition of $\hat{\alpha}(S,(1-\beta)C)$
in \S\ref{subsection:P2}.
\end{proof}

\begin{lemma}
\label{lemma:quadric} Suppose that
$S\cong\mathbb{P}^1\times\mathbb{P}^1$. Then
\eqref{equation:log-pair} is log canonical at $P$.
\end{lemma}

\begin{proof}
Suppose that \eqref{equation:log-pair} is not log canonical at
$P$. Let $L_1$ and $L_2$ be the fibers of two different
projections $S\to\mathbb{P}^1$ that both pass through $P$. Since
$(S,(1-\beta)C+\hat{\alpha}(S,(1-\beta)C)\beta(2L_1+2L_2))$ is log
canonical and $2L_1+2L_2\sim_{\mathbb{R}} D$, we may assume that
either $L_1\not\subset\mathrm{Supp}(D)$ or
$L_2\not\subset\mathrm{Supp}(D)$ by Lemma~\ref{lemma:convexity}.
This implies that $\mathrm{mult}_{P}(D)\leqslant 2$, since $D\cdot
L_1=D\cdot L_2=2$. Then
$$
\hat{\alpha}(S,(1-\beta)C)\beta\mathrm{mult}_{P}(D)\leqslant 2\hat{\alpha}(S,(1-\beta)C)\beta\leqslant \min\Big\{1,\frac{1}{4}+\beta\Big\},%
$$
(see
\S\ref{subsection:quadric}). Applying
Corollary~\ref{corollary:2012-summer} with $n=4$, we get
$$
8\hat{\alpha}(S,(1-\beta)C)\beta=\hat{\alpha}(S,(1-\beta)C)\beta(C\cdot D)\geqslant\hat{\alpha}(S,(1-\beta)C)\beta\mathrm{mult}_{P}\Big(C\cdot D\Big)>1+4\beta,%
$$
which contradicts the definition of $\hat{\alpha}(S,(1-\beta)C)$
in \S\ref{subsection:quadric}.
\end{proof}

\begin{lemma}
\label{lemma:1-2-3} Suppose that $K_{S}^2\leqslant 3$. Then
\eqref{equation:log-pair} is log canonical at $P$.
\end{lemma}

\begin{proof}
Suppose that \eqref{equation:log-pair} is not log canonical at
$P$. By \cite[Theorem~1.12]{ChePark13}, there is $T\in|-K_{S}|$
such that $(S,T)$ is not log canonical at $P$, and \emph{all}
irreducible components of the curve $T$ are contained in the
support of the divisor $D$. Moreover, such $T$ is unique.

Since $(S,T)$ is not log canonical at $P$, we have very limited
number of choices for $T\in|-K_{S}|$. Going through all of them,
we see that $(S,(1-\beta)C+\hat{\alpha}(S,(1-\beta)C)\beta T)$ is
log canonical at $P$ (for details, see the proofs of
\cite[Theorems~4.9.1, 4.10.1, 4.11.1]{JMG-thesis}).

By Lemma~\ref{lemma:convexity}, there is an effective
$\mathbb{R}$-divisor $D^\prime$ on the surface $S$ such that
$D^\prime\sim_{\mathbb{R}} D$, the log pair
$(S,(1-\beta)C+\hat{\alpha}(S,(1-\beta)C)\beta D^\prime)$ is not
log canonical at $P$, and $\mathrm{Supp}(D^\prime)$ does not
contain at least one irreducible component of $T$. The latter
contradicts \cite[Theorem~1.12]{ChePark13}.
\end{proof}

\begin{corollary}
\label{corollary:1-2-3} Theorem~\ref{theorem:main-strong} holds in
the following cases: $S\cong\mathbb{P}^2$,
$S\cong\mathbb{P}^1\times\mathbb{P}^1$ and $K_{S}^2\leqslant 3$.
\end{corollary}

\begin{lemma}
\label{lemma:P-in-two-lines} Suppose that $4\leqslant
K_{S}^2\leqslant 7$, and $P$ is the intersection point of two
intersecting $(-1)$-curves in $S$. Then \eqref{equation:log-pair}
is log canonical at $P$.
\end{lemma}

\begin{proof}
Suppose that \eqref{equation:log-pair} is not log canonical at
$P$. Denote by $Z_1$ and $Z_2$ two $(-1)$-curves in $S$ that
contains $P$. We write $D=aZ_1+bZ_2+\Omega$, where $a$ and $b$ are
non-negative real numbers, and $\Omega$ is an effective
$\mathbb{R}$-divisor that whose support does not contain  $Z_1$
and $Z_2$. By Lemma~\ref{lemma:P-in-C}, one has $a>0$ and $b>0$.
Let $x=\mathrm{mult}_{P}(\Omega)$. Then $1-b+a=\Omega\cdot
Z_1\geqslant x$, which gives $b-a+x\leqslant 1$. Similarly, we
obtain $a-b+x\leqslant 1$. Then $a\leqslant 1+b$, $b\leqslant 1+a$
and $x\leqslant 1$. Thus, we have
$$
\mathrm{mult}_{P}\Big((1-\beta)C+\hat{\alpha}(S,(1-\beta)C)\beta\Omega\Big)=1-\beta+\hat{\alpha}(S,(1-\beta)C)\beta x\leqslant 1-\beta+\hat{\alpha}(S,(1-\beta)C)\beta\leqslant 1,%
$$
because $\hat{\alpha}(S,(1-\beta)C)\leqslant 1$. Applying
Theorem~\ref{theorem:Trento} to \eqref{equation:log-pair}, we see
that
$$
2\Big(1-\hat{\alpha}(S,(1-\beta)C\Big)\beta
a)<Z_1\cdot\Big(\hat{\alpha}(S,(1-\beta)C)\beta\Omega+(1-\beta)C\Big)=\hat{\alpha}(S,(1-\beta)C)\beta(1-a+b)+1-\beta,
$$
or
$$
2\Big(1-\hat{\alpha}(S,(1-\beta)C)\beta b\Big)<Z_2\cdot
\Big(\hat{\alpha}(S,(1-\beta)C)\beta\Omega+(1-\beta)C\Big)=\hat{\alpha}(S,(1-\beta)C)\beta\Big(1-b+a\Big)+1-\beta.
$$
In both cases, we obtain
$\hat{\alpha}(S,(1-\beta)C)\beta(1+a+b)>1+\beta$.

Suppose that $K_S^2=7$. Let us use the notation of
\S\ref{subsection:deg-7}. We may assume that $Z_1=E_1$ and
$Z_2=L$. Since $3L+2E_1+2E_2\sim -K_{S}$ and
$(S,(1-\beta)C+\hat{\alpha}(S,(1-\beta)C)\beta(3L+2E_1+2E_2)$ is
log canonical, we may also assume that
$E_2\not\subset\mathrm{Supp}(\Omega)$ by
Lemma~\ref{lemma:convexity}. Then $1-b=E_2\cdot\Omega\geqslant 0$,
which gives $b\leqslant 1$. Since $a\leqslant 1+b$, we get
$a+b\leqslant 3$. Thus, we have
$$
4\beta\hat{\alpha}(S,(1-\beta)C)\geqslant\hat{\alpha}(S,(1-\beta)C)\beta\Big(1+a+b\Big)>1+\beta,
$$
which contradicts the definition of $\hat{\alpha}(S,(1-\beta)C)$.

Suppose that $K_S^2=6$. Let us use the notation of
\S\ref{subsection:deg-6}. Without loss of generality, we may
assume that $Z_1=E_1$ and $Z_2=L_{12}$. Since
$(S,(1-\beta)C+\hat{\alpha}(S,(1-\beta)C)\beta(2L_{12}+2E_1+L_{13}+E_2))$
is log canonical and $2L_{12}+2E_1+L_{13}+E_2\sim -K_{S}$, we may
assume that $\mathrm{Supp}(\Omega)$ does not contain $L_{13}$ or
$E_2$ by Lemma~\ref{lemma:convexity}. If
$L_{13}\not\subset\mathrm{Supp}(\Omega)$, then $1-a=\Omega\cdot
L_{13}\geqslant 0$, which implies that $a\leqslant 1$. Similarly,
if $E_2\not\subset\mathrm{Supp}(\Omega)$, then $b\leqslant 1$.
Since $a\leqslant 1+b$ and $b\leqslant 1+a$, we see that
$a+b\leqslant 3$. Thus, we have
$$
4\beta\hat{\alpha}(S,(1-\beta)C)\geqslant\hat{\alpha}(S,(1-\beta)C)\beta\Big(1+a+b\Big)>1+\beta,
$$
which contradicts the definition of $\hat{\alpha}(S,(1-\beta)C)$.

Suppose that $K_S^2=5$. Let us use the notation of
\S\ref{subsection:deg-5}. Without loss of generality, we may
assume that $Z_1=E_1$ and $Z_2=L_{12}$. Since
$(S,(1-\beta)C+\hat{\alpha}(S,(1-\beta)C)\beta(2E_1+L_{12}+L_{13}+L_{14}))$
is log canonical and $2E_1+L_{12}+L_{13}+L_{14}\sim -K_S$, we may
assume that $\mathrm{Supp}(\Omega)$ does not contain $L_{13}$ or
$L_{14}$ by Lemma~\ref{lemma:convexity}. Since
 $(S,
(1-\beta)C+\hat{\alpha}(S,(1-\beta)C)\beta(E_1+2L_{12}+E_2+L_{34}))$
is log canonical and $E_1+2L_{12}+E_2+L_{34}\sim -K_S$, we may
assume that $\mathrm{Supp}(\Omega)$ does not contain  $E_2$ or
$L_{34}$   by Lemma~\ref{lemma:convexity}. If
$L_{13}\not\subset\mathrm{Supp}(\Omega)$, then $1-a=\Omega\cdot
L_{13}\geqslant 0$, which gives $a\leqslant 1$. Similarly, if
$L_{14}\not\subset\mathrm{Supp}(\Omega)$, then $a\leqslant 1$. If
$E_{2}\not\subset\mathrm{Supp}(\Omega)$, then $1-b=\Omega\cdot
E_{2}\geqslant 0$, which gives $b\leqslant 1$. Similarly, if
$L_{34}\not\subset\mathrm{Supp}(\Omega)$, then $b\leqslant 1$.
Thus, we have $a\leqslant 1$ and $b\leqslant 1$. Then
$$
3\beta\hat{\alpha}(S,(1-\beta)C)\geqslant\hat{\alpha}(S,(1-\beta)C)\beta\Big(1+a+b\Big)>1+\beta,
$$
which contradicts the definition of $\hat{\alpha}(S,(1-\beta)C)$.

We have $K_{S}^2=4$. Let us use the notation of
\S\ref{subsection:deg-4}. Without loss of generality, we may
assume that $Z_1=L_{12}$ and $Z_2=L_{34}$. Let $Z$ be the proper
transform on $S$ of the line in $\mathbb{P}^2$ that passes through
$\pi(E_5)$ and $\pi(L_{12}\cap L_{34})$. Since
$(S,(1-\beta)C+\hat{\alpha}(S,(1-\beta)C)\beta(L_{12}+L_{34}+Z))$
is log canonical and $L_{12}+L_{34}+Z\sim -K_{S}$, we may assume
that $Z\not\subset\mathrm{Supp}(\Omega)$ by
Lemma~\ref{lemma:convexity}. Then $2-a-b=\Omega\cdot Z\geqslant
0$, which implies that
$3\beta\hat{\alpha}(S,(1-\beta)C)\geqslant\hat{\alpha}(S,(1-\beta)C)\beta(1+a+b)>1+\beta$.
The latter contradicts the definition of
$\hat{\alpha}(S,(1-\beta)C)$.
\end{proof}

\begin{lemma}
\label{lemma:F1} Suppose $S\cong\mathbb{F}_1$, and $P$ is
contained in a unique $(-1)$-curve in $S$. Then
\eqref{equation:log-pair} is log canonical at $P$.
\end{lemma}

\begin{proof}
Let us use the notation of \S\ref{subsection:F1}. Then $P=Z\cap
C$, since $P\in C$. Suppose that \eqref{equation:log-pair} is not
log canonical at $P$.  By Lemma~\ref{lemma:P-in-C}, we have
$Z\subset\mathrm{Supp}(D)$. By Lemma~\ref{lemma:convexity}, we may
assume that $F\not\subset\mathrm{Supp}(D)$, since
$(S,(1-\beta)C+\hat{\alpha}(S,(1-\beta)C)\beta(2Z+3F))$ is log
canonical and $2Z+3F\sim -K_{S}$. Then
$\mathrm{mult}_{P}(D)\leqslant F\cdot D=2$. On the other hand, we
have $2\hat{\alpha}(S,(1-\beta)C)\beta\leqslant\frac{1}{4}+\beta$
and $2\hat{\alpha}(S,(1-\beta)C)\beta\leqslant 1$. Applying
Corollary~\ref{corollary:2012-summer} with $n=4$ to
 \eqref{equation:log-pair}, we get
$$
8\hat{\alpha}(S,(1-\beta)C)\beta=\hat{\alpha}(S,(1-\beta)C)\beta(C\cdot D)\geqslant\hat{\alpha}(S,(1-\beta)C)\beta\mathrm{mult}_{P}\Big(C\cdot D\Big)>1+4\beta,%
$$
which contradicts the definition of $\hat{\alpha}(S,(1-\beta)C)$.
\end{proof}

\begin{lemma}
\label{lemma:P-in-one-line} Suppose that $4\leqslant
K_{S}^2\leqslant 7$, and $P$ is contained in a $(-1)$-curve in
$S$. Then \eqref{equation:log-pair} is log canonical at $P$.
\end{lemma}

\begin{proof}
See Section~\ref{section:degree-7}.
\end{proof}

The following result implies
Corollary~\ref{corollary:delPezzos-blow-ups} \emph{modulo}
Theorem~\ref{theorem:main-strong}.

\begin{theorem}
\label{theorem:delPezzos-blow-ups} Let $S_1$ and $S_2$ be smooth
del Pezzo surfaces, let $C_1$ and $C_2$ be smooth curves in
$|-K_{S_1}|$ and $|-K_{S_2}|$, respectively. Suppose that there
exists a birational morphism $f\colon S_2\to S_1$ such
that~$f(C_2)=C_1$. Then $\hat{\alpha}(S_1,
(1-\beta)C_1)\leqslant\hat{\alpha}(S_2, (1-\beta)C_2)$ for every
$\beta\in(0,1]$ except the following cases:
\begin{enumerate}
\item $S_1\cong\mathbb{P}^2$, $S_2\cong\mathbb{F}_1$, and $f$ is the blow up of an inflection points of the cubic curve $C_1\subset\mathbb{P}^2$,%

\item $S_1\cong\mathbb{P}^1\times\mathbb{P}^1$, $K_{S_2}^2=7$, and $f$ is the blow up of a point in $C_1$.%
\end{enumerate}
\end{theorem}

\begin{proof}
Since $f(C_2)=C_1$, the morphism $f$ is the blow up of
$K_{S_1}^2-K_{S_2}^2\geqslant 0$ distinct points on the curve
$C_2$. Suppose that $\hat{\alpha}(S_1,
(1-\beta)C_1)>\hat{\alpha}(S_2, (1-\beta)C_2)$. Going through all
possible cases considered in Section~\ref{section:formulas}, we
end up with the following possibilities:
\begin{enumerate}
\item $S_1\cong\mathbb{P}^2$, $S_2\cong\mathbb{F}_1$, and $f$ is the blow up of an inflection points of the cubic curve $C_1\subset\mathbb{P}^2$,%

\item $S_1\cong\mathbb{P}^1\times\mathbb{P}^1$, $K_{S_2}^2=7$, and $f$ is the blow up of a point in $C_1$,%

\item $K_{S_1}^2=4$, $K_{S_2}^2=3$, the morphism $f$ is the blow up
of a point in $C_1$, the curve $C_1$ does not contain intersection
points of any two lines, for every two conics $Z_1$ and $Z_2$ in
$S_1$ such that $Z_1+Z_2\sim -K_{S_1}$, the conics $Z_1$ and $Z_2$
do not tangent
$C_1$ at one point, and $S_2$ contains an Eckardt point and this point is contained in $C_2$,%

\item $K_{S_1}^2=3$, $K_{S_2}^2=2$, the morphism $f$ is the blow up
of a point in $C_1$, the surface $S_1$ contains no Eckardt points,
for every line $L$ and every conic $M$ on $S_1$ such that $L$ is
tangent to $M$ we have $L\cap M\not\in C_1$, and every irreducible
cuspidal curve $T\in|-K_{S_1}|$ intersects $C_1$ by at least two
point, the linear system $|-K_{S_2}|$ contains a curve with a
tacnodal singularity and this tacnodal singular point is contained
in $C_2$.
\end{enumerate}
The first two cases are indeed possible. Let us show that the last
two cases are impossible. Denote by $E$ the $f$-exceptional curve.
Then $f(E)\in C_1$.

Suppose that $K_{S_1}^2=4$ and $K_{S_2}^2=3$. Then $C_2$ contains
an Eckardt point $O$. Denote by $L_1,L_2,L_3$ the lines in $S_2$
that passes through $O$. Then either $E$ is one of these three
lines, or $E$ intersects exactly one of them. Without loss of
generality, we may assume that either $E=L_3$ or $E\cap
L_{1}=E\cap L_{3}=\varnothing$. In the former case, $f(L_1)$ and
$f(L_2)$ are two conics in $S_1$ such that $f(L_1)+
f(L_2)\sim-K_{S_2}$, and both $f(L_1)$ and $f(L_2)$ tangent the
curve $C_1=f(C_2)$ at the point $f(P)\in C_1$. Since we know that
such conics do not exist by assumption, we conclude that $E\cap
L_{1}=E\cap L_{3}=\varnothing$. Then $f(L_1)$ and $f(L_2)$ are two
lines in $S_1$ that both pass through the point $f(P)\in C_1$.
Such lines do not exist either. Thus, this case is impossible.

Now we suppose that $K_{S_1}^2=3$ and $K_{S_2}^2=2$. Let $Z$ be a
curve in $|-K_{S_2}|$ such that $Z$ has tacnodal singularity $Q\in
C_2$. Then $Z=L_1+L_2$, where $L_1$ and $L_2$ are two
$(-1)$-curves in $S_2$ that are tangent each other at the point
$Q\in C_2$. Then either $E$ is one of these two curves, or $E$
intersects exactly one of them. Without loss of generality, we may
assume that either $E=L_2$ or $E\cap L_{1}=\varnothing$. In the
former case, $f(L_1)$ is a cuspidal curve in $|-K_{S_1}|$ whose
intersection with the curve $C_1$ consists of the point
$f(Q)=\mathrm{Sing}(f(L_1))$. By assumption, such a cuspidal curve
does not exist. Thus, $E\cap L_{1}=\varnothing$. Then $f(L_1)$ is
a line, and $f(L_2)$ is a conic. Moreover, the line $f(L_1)$
tangents to $f(L_2)$ at the point $f(Q)\in C_1$. The latter is
impossible by assumption.
\end{proof}

To prove Theorem~\ref{theorem:main-strong}, we have to prove that
\eqref{equation:log-pair} is log canonical at $P$, where $P$ is a
point in $C\not\subset\mathrm{Supp}(D)$. The latter follows from
Corollary~\ref{corollary:1-2-3},
Lemmas~\ref{lemma:P-in-two-lines}, \ref{lemma:F1},
\ref{lemma:deg-4-line}, \ref{lemma:P-in-one-line} and

\begin{lemma}
\label{lemma:induction} Suppose that $K_{S}^2\geqslant 3$, and
neither $S\cong\mathbb{P}^2$ nor
$S\cong\mathbb{P}^1\times\mathbb{P}^1$. Suppose that $P$ is not
contained in any $(-1)$-curve in $S$. If
Theorem~\ref{theorem:main-strong} holds for all smooth del Pezzo
surfaces of degree $K_{S}^2-1$, then \eqref{equation:log-pair} is
log canonical at $P$.
\end{lemma}

\begin{proof}
Suppose that \eqref{equation:log-pair} is not log canonical at
$P$. Let $f\colon\tilde{S}\to S$ be a blow up of  $P$. Then
$\tilde{S}$ is a smooth del Pezzo surface of degree
$K_{\tilde{S}}^2=K_{S}^2-1$, since $P$ is not contained in any
$(-1)$-curve in $S$.  Denote the $f$-exceptional curve by $E$,
denote the proper transform of $C$ on $\tilde{S}$ by $\tilde{C}$,
and denote the proper transform of $D$ on $\tilde{S}$  by
$\tilde{D}$. Then $\tilde{C}\in |-K_{\tilde{S}}|$, since $P\in C$.
The log pair
\begin{equation}
\label{equation:log-pair-inductive}
\Big(\tilde{S},(1-\beta)\tilde{C}+\hat{\alpha}(S,
(1-\beta)C)\beta\Big(\tilde{D}+\Big(\mathrm{mult}_{P}(D)-\frac{1}{\hat{\alpha}(S,(1-\beta)C)}\Big)E\Big)\Big)%
\end{equation}
is not log canonical by Lemma~\ref{lemma:log-pull-back}. Let
$\tilde{D}^\prime=\tilde{D}+(\mathrm{mult}_{P}(D)-1)E$. Then
$\tilde{D}^\prime\sim_{\mathbb{R}} -K_{\tilde{S}}$, and
$\tilde{D}^\prime$ is effective by Lemma~\ref{lemma:P-in-C}.
Furthermore, the log pair
$(\tilde{S},(1-\beta)\tilde{C}+\hat{\alpha}(S,
(1-\beta)C)\beta\tilde{D}^\prime)$ is not log canonical, because
\eqref{equation:log-pair-inductive} is not log canonical. This
shows that $\hat{\alpha}(S, (1-\beta)C)>\alpha(\tilde{S},
(1-\beta)\tilde{C})$. But it follows from
Theorem~\ref{theorem:delPezzos-blow-ups} that
$\hat{\alpha}(\tilde{S}, (1-\beta)\tilde{C})\geqslant
\hat{\alpha}(S, (1-\beta)C)$. Thus, we see that
$\hat{\alpha}(\tilde{S}, (1-\beta)\tilde{C})>\alpha(\tilde{S},
(1-\beta)\tilde{C})$. Hence, Theorem~\ref{theorem:main-strong}
does not hold for $\tilde{S}$.
\end{proof}

This completes the proof of Theorem~\ref{theorem:main-strong}
\emph{modulo} Lemma~\ref{lemma:P-in-one-line}.

\section{The proof of Lemma~\ref{lemma:P-in-one-line}}
\label{section:degree-7}

In this section, we will prove Lemma~\ref{lemma:P-in-one-line}.
Let us use its notation and assumptions. Then $4\leqslant
K_{S}^2\leqslant 7$ and $P$ is a point in
$C\not\subset\mathrm{Supp}(D)$ that is contained in a $(-1)$-curve
in $S$. Let us denote this $(-1)$-curve by $\mathcal{L}$. We must
prove that \eqref{equation:log-pair} is log canonical at $P$. By
Lemma~\ref{lemma:P-in-two-lines}, we may assume that $\mathcal{L}$
is the only $(-1)$-curve in $S$ that contains $P$. We write
$D=a\mathcal{L}+\Omega$, where $a$ is a non-negative real number,
and $\Omega$ is an effective $\mathbb{R}$-divisor such that
$\mathcal{L}\not\subset\mathrm{Supp}(\Omega)$. By
Lemma~\ref{lemma:P-in-C}, we have $a>0$. Let
$x=\mathrm{mult}_{P}(\Omega)$. Then $1+a=\mathcal
L\cdot\Omega\geqslant x$.

\begin{corollary}
\label{corollary:x-a-1} One has $x\leqslant 1+a$.
\end{corollary}

Let $\lambda=\hat{\alpha}(S,(1-\beta)C)$. Consider a sequence of
$4$ blow ups
$$
\xymatrix{S_4\ar@{->}[rr]^{\pi_4}&& S_{3}\ar@{->}[rr]^{\pi_{3}}&&
S_2\ar@{->}[rr]^{\pi_2}&& S_1\ar@{->}[rr]^{\pi_1}&&S}
$$
such that $\pi_1$ is the blow up of the point $P$, $\pi_2$ is the
blow up of the intersection point of the $\pi_1$-exceptional curve
and the proper transform of the curve $C$ on $S_1$, $\pi_3$ is the
blow up of the intersection point of the $\pi_2$-exceptional curve
and the proper transform of the curve $C$ on $S_2$, and $\pi_4$ is
the blow up of the intersection point of the $\pi_3$-exceptional
curve and the proper transform of the curve $C$ on $S_3$. Denote
by $F_1$, $F_2$, $F_3$ and $F_4$ the exceptional curves of the
blow ups $\pi_1$, $\pi_2$, $\pi_3$ and $\pi_4$, respectively.
Denote by $C^1$, $C^2$, $C^3$ and $C^4$ the proper transforms of
the curve $C$ on the surfaces $S_{1}$, $S_{2}$, $S_{3}$ and
$S_{4}$, respectively. Let $P_1=C^1\cap F_1$, $P_2=C^2\cap F_2$,
$P_3=C^3\cap F_3$ and $P_4=C^4\cap F_4$. Denote the proper
transform of the divisor $\Omega$ on the surfaces $S_1$, $S_2$,
$S_3$ and $S_4$ by $\Omega^1$, $\Omega^2$, $\Omega^3$ and
$\Omega^4$, respectively. Let $x_1=\mathrm{mult}_{P_1}(\Omega)$,
$x_2=\mathrm{mult}_{P_2}(\Omega)$ and
$x_3=\mathrm{mult}_{P_3}(\Omega)$.

\begin{lemma}
\label{lemma:deg-7-4-blow-ups} Suppose that
\eqref{equation:log-pair} is not log canonical at $P$. Then at
least one of the following four conditions is not satisfied:
\begin{enumerate}
\item[(i)] $\lambda\beta (a+x)\leqslant 1$,%
\item[(ii)] $2\lambda\beta (a+x)-2\beta\leqslant 1$ or $\lambda\beta (a+x+x_1)-\beta\leqslant 1$,%
\item[(iii)] $\lambda\beta (a+x+2x_1)-3\beta\leqslant 1$ or $\lambda\beta (a+x+x_1+x_2)-2\beta\leqslant 1$,%
\item[(iv)] $\lambda\beta (a+x+x_1+2x_2)-4\beta\leqslant 1$ or $\lambda\beta (a+x+x_1+x_2+x_3)-3\beta\leqslant 1$.%
\end{enumerate}
If $\lambda\beta K_{S}^2\leqslant 1+3\beta$, then at
least one of the conditions (i), (ii) or (iii) is not satisfied.
\end{lemma}

\begin{proof}
If conditions (i), (ii), (iii) and (iv) are satisfied, then
Corollary~\ref{corollary:4-blow-ups} gives
$$
K_{S}^2=D\cdot C\geqslant\mathrm{mult}_{P}\Big(D\cdot C\Big)>\frac{1+4\beta}{\lambda\beta},%
$$
which is impossible, since $\lambda\beta K_{S}^2\leqslant
1+4\beta$ by the definition of $\lambda=\hat\alpha(S, (1-\beta)C)$
for $4\leqslant K_S^2\leqslant 7$.  Similarly, if conditions (i),
(ii), (iii) are satisfied, then $\lambda\beta K_{S}^2>1+3\beta$ by
Corollary~\ref{corollary:4-blow-ups}.
\end{proof}

\begin{lemma}
\label{lemma:deg-7} Suppose that $K_{S}^2=7$. Then
\eqref{equation:log-pair}  is log canonical at $P$.
\end{lemma}

\begin{proof}
Suppose that \eqref{equation:log-pair} is not log canonical at
$P$. Let us use the notation of \S\ref{subsection:deg-7}. Without
loss of generality, we may assume that either $\mathcal{L}=E_1$ or
$\mathcal{L}=L$ (but not both).

Suppose that $\mathcal{L}=L$. Since $P\not\in E_1\cup E_2$, the
curve $R$ is smooth and irreducible. Since
$(S,(1-\beta)C,\lambda\beta(L+2R))$ is log canonical and $L+2R\sim
-K_{S}$, we may assume that $R\not\subset\mathrm{Supp}(\Omega)$.
Denote the proper transform of the curve $R$ on $S_1$ by $R^1$,
and denote its proper transform on $S_2$ by $R^2$. Then
$3-a-x-x_1=R^2\cdot\Omega^2\geqslant 0$, which gives
$a+x+x_1\leqslant 3$. Since $x-a\leqslant 1$ by
Corollary~\ref{corollary:x-a-1}, then $x_1\leqslant \frac{4}{3}$
and all conditions of Lemma ~\ref{lemma:deg-7-4-blow-ups} are
satisfied, giving a contradiction.

We have $\mathcal{L}=E_1$. Then $L_1$ is irreducible, since
$P\not\in L$. Since $(S,(1-\beta)C,\lambda\beta(2L_1+2E_1+L))$ is
log canonical  and $2L_1+2E_1+L\sim -K_{S}$, we may assume that
 $L_1$ or $L$ is not contained in
$\mathrm{Supp}(\Omega)$ by Lemma~\ref{lemma:convexity}. We write
$\Omega=bL_1+\Delta$, where $b$ is a non-negative real number, and
$\Delta$ is an effective $\mathbb{R}$-divisor on $S$ such that
$L_1\not\subset\mathrm{Supp}(\Delta)$ and
$E_1\not\subset\mathrm{Supp}(\Delta)$. Then
$1-b+a=E_1\cdot\Delta\geqslant y$, which gives $b+y\leqslant 1+a$.
If $b>0$, then $a\leqslant 1$. Indeed, if
$L\not\subset\mathrm{Supp}(\Delta)$, then
$1-a=L\cdot\Delta\geqslant 0$.

Denote the proper transform of the divisor $\Delta$ on $S_1$ by
$\Delta^1$ , denote  the proper transform of the divisor $\Delta$
on $S_2$ by $\Delta^2$, and denote the proper transform of the
divisor $\Delta$ on $S_3$ by $\Delta^3$ . Let
$y=\mathrm{mult}_{P}(\Delta)$,
$y_1=\mathrm{mult}_{P_1}(\Delta^1)$,
$y_2=\mathrm{mult}_{P_2}(\Delta^2)$ and
$y_3=\mathrm{mult}_{P_3}(\Delta^3)$. Then  $x=b+y$. Since
$L_1\cdot C=2$, either $\mathrm{mult}_{P}(L_1\cdot C)=1$ or
$\mathrm{mult}_{P}(L_1\cdot C)=2$. Thus, we have, $x_2=y_2$ and
$x_3=y_3$.

Suppose that $\mathrm{mult}_{P}(L_1\cdot C)=1$. Then $x_1=y_1$ and
$2-a=L_1\cdot\Delta\geqslant y$. We have $b+y\leqslant 1+a$ by
Corollary~\ref{corollary:x-a-1}. If $b>0$, then $a\leqslant 1$.
Therefore, we have $\lambda\beta (a+x)\leqslant 1$, $\lambda\beta
(a+x+x_1)-\beta\leqslant 1$, $\lambda\beta
(a+x+2x_1)-3\beta\leqslant 1$ and $\lambda\beta
(a+x+x_1+2x_2)-4\beta\leqslant 1$, which contradicts
Lemma~\ref{lemma:deg-7-4-blow-ups}.

Thus we see that $\mathrm{mult}_{P}(L_1\cdot C)=2$. Then
$x_1=y_1+b$ and $2-a=L_1\cdot\Delta\geqslant y+y_1$, which gives
$a+y+y_1\leqslant 2$. Since $L_1$ is tangent to $C$ at the point
$P$, we have
$$
\lambda=\hat{\alpha}\big(S,(1-\beta)C\big)\leqslant\mathrm{min}\Big\{1,
\frac{1+2\beta}{7\beta}, \frac{1}{3\beta}\Big\}.
$$
Moreover, we have $b+y\leqslant 1+a$ by
Corollary~\ref{corollary:x-a-1}. Furthermore, if $b>0$, then
$a\leqslant 1$. This gives $\lambda\beta (a+x)\leqslant 1$,
$2\lambda\beta (a+x)-2\beta\leqslant 1$, $\lambda\beta
(a+x+x_1+x_2)-2\beta\leqslant 1$ and $\lambda\beta
(a+x+x_1+2x_2)-4\beta\leqslant 1$, which is impossible by
Lemma~\ref{lemma:deg-7-4-blow-ups}.
\end{proof}

\begin{lemma}
\label{lemma:deg-6} Suppose that $K_{S}^2=6$. Then
\eqref{equation:log-pair}  is log canonical at $P$.
\end{lemma}

\begin{proof}
Suppose that \eqref{equation:log-pair} is not log canonical at
$P$. Let us use the notation of \S\ref{subsection:deg-6}. Without
loss of generality, we may assume that $\mathcal{L}=E_1$. Denote
the proper transform of the curve $E_1$ on the surface $S_1$ by
$E_1^1$. Let $L$ be the proper transform on $S$ of the line in
$\mathbb{P}^2$ that is tangent to $\pi(C)$ at the point $\pi(P)$.
Then $-K_{S}\cdot L=2$, since $P\not\in L_{12}\cup L_{13}\cup
L_{23}$. Denote the proper transform of the curve $L$ on  $S_1$ by
$L^1$, denote the proper transform of the curve $L$ on   $S_2$ by
$L^2$, and denote the proper transform of the curve $L$ on   $S_3$
by $L^3$.

We claim that $L\subset\mathrm{Supp}(\Omega)$. Indeed, suppose
that $L\not\subset\mathrm{Supp}(\Omega)$. Then $a+x\leqslant 2$,
since $2-a=\Omega\cdot L\geqslant x$. But $x\leqslant 1+a$ by
Corollary~\ref{corollary:x-a-1}. Therefore, we have $x_1\leqslant
x\leqslant\frac{3}{2}$. These inequalities give $\lambda\beta
(a+x)\leqslant 1$, $2\lambda\beta(a+x)-\beta\leqslant 1$ and
$\lambda\beta(a+x+2x_1)-3\beta\leqslant 1$. Therefore,
$\lambda\beta (a+x+x_1+2x_2)-4\beta>1$ and
$6\lambda\beta>1+3\beta$ by Lemma~\ref{lemma:deg-7-4-blow-ups}.
The former inequality implies that $a+x+x_1+2x_2>6$. The latter
inequality implies that $L$ is not tangent to $C$ at the point $P$
(see \S\ref{subsection:deg-6}).

Let $Z$ be the proper transform on $S$ of the conic in
$\mathbb{P}^2$ that passes through the points $\pi(E_1)$,
$\pi(E_2)$, $\pi(E_3)$, and is tangent to $\pi(C)$ at the point
$\pi(P)$. Then $Z$ is irreducible, $E_1+L+Z\sim -K_{S}$ and
$-K_{S}\cdot Z=3$, since $L$ is not tangent to $C$ at $P$. Then
$\mathrm{mult}_{P}(Z\cdot C)\leqslant 3$, since $-K_{S}\cdot Z=3$.

We write $\Omega=cZ+\Upsilon$, where $c$ is a non-negative real
number, and $\Upsilon$ is an effective $\mathbb{R}$-divisor on $S$
whose support does not contain $Z$. Denote the proper transform of
the divisor $\Upsilon$ on $S_1$ by $\Upsilon^1$, denote the proper
transform of the divisor $\Upsilon$ on $S_2$ by $\Upsilon^2$, and
denote the proper transform of the divisor $\Upsilon$ on $S_3$ by
$\Upsilon^3$. Let $z=\mathrm{mult}_{P}(\Upsilon)$,
$z_1=\mathrm{mult}_{P_1}(\Upsilon^1)$,
$z_2=\mathrm{mult}_{P_2}(\Upsilon^2)$,
$z_3=\mathrm{mult}_{P_3}(\Upsilon^3)$. Then $x=c+z$, $x_1=c+z_1$,
$x_3=z_3$. If $\mathrm{mult}_{P}(Z\cdot C)=2$, then $x_2=z_2$ and
$3-a-c-z=Z^1\cdot
\Upsilon^1\geqslant\mathrm{mult}_{P_1}(Z^1\cdot\Upsilon^1)\geqslant
z_1$, which implies that
$$
6<a+x+x_1+2x_2=a+z+z_1+2z_2+2c\leqslant3+2z_2+c\leqslant
3+2z_2+2c\leqslant 3+2x\leqslant 6,
$$
since $z+c\leqslant \frac{3}{2}$ and $a+c+z\leqslant 2$. Thus, we
see that $\mathrm{mult}_{P}(Z\cdot C)=3$. Then $x_2=c+z_2$ and
$3-a-c-z-z_1=Z^2\cdot\Upsilon^2\geqslant\mathrm{mult}_{P_2}(Z^2\cdot
\Upsilon^2)\geqslant z_2$, which gives $a+c+z+z_1+z_2\leqslant 3$.
Then
$$
6<a+x+x_1+2x_2=a+z+z_1+2z_2+3c<3+z_2+2c\leqslant
3+2z_2+2c\leqslant 3+2x\leqslant 6,
$$
which is absurd. This shows that $L\subset\mathrm{Supp}(\Omega)$.

We write $\Omega=bL+\Delta$, where $b$ is a positive real number,
and $\Delta$ is an effective $\mathbb{R}$-divisor on $S$ such that
$L\not\subset\mathrm{Supp}(\Delta)$. Let
$y=\mathrm{mult}_{P}(\Delta)$. Then $2-a=\Delta\cdot L\geqslant
y$. Denote the proper transform of the divisor $\Delta$ on $S_1$
by $\Delta^1$, denote the proper transform of the divisor $\Delta$
on $S_2$ by $\Delta^2$, and denote the proper transform of the
divisor $\Delta$ on $S_3$ by $\Delta^3$. Let
$y_1=\mathrm{mult}_{P_1}(\Delta^1)$,
$y_2=\mathrm{mult}_{P_2}(\Delta^2)$ and
$y_3=\mathrm{mult}_{P_3}(\Delta^3)$. Then $x=b+y$, $x_2=y_2$ and
$x_3=y_3$, which implies that $b+y\leqslant 1+a$ by
Corollary~\ref{corollary:x-a-1}. Then
\begin{equation}
\label{equation:deg-6-D-S-1} \Big(S_1,(1-\beta)C^1+\lambda\beta aE_1^1+\lambda\beta bL^1+\lambda\beta\Delta^1+\Big(\lambda\beta (a+b+y)-\beta\Big)F_1\Big)%
\end{equation}
is not log canonical at some point $Q_1\in F_1$ by
Lemma~\ref{lemma:log-pull-back}.

We claim that $L$ is tangent to $C$ at the point $P$. Indeed,
suppose that $L$ is not tangent to $C$ at $P$. Then $x_1=y_1$. Let
$Z$ be the proper transform on $S$ of the conic in $\mathbb{P}^2$
that passes through $\pi(E_1)$, $\pi(E_2)$, $\pi(E_3)$, and is
tangent to $\pi(C)$ at $\pi(P)$. Then $Z$ is irreducible and
$-K_{S}\cdot Z=3$. Moreover, we have $E_1+L+Z\sim -K_{S}$, and the
log pair $(S, (1-\beta)C+\lambda\beta(E_1+L+Z))$ is log canonical.
Thus, we may assume that $Z\not\subset\mathrm{Supp}(D)$ by
Lemmas~\ref{lemma:convexity}. Then $3-a-b-y=Z^1\cdot
\Delta^1\geqslant\mathrm{mult}_{P_1}(Z^1\cdot\Delta^1)\geqslant
y_1$. Since we also have $b+y\leqslant 1+a$, $a+y\leqslant 2$,
$x=y+b$, $x_1=y_1$ and $x_2=y_2$, we see that
\begin{equation}
\label{equation:deg-6-L-tangent-C}
\left.\aligned%
&\lambda\beta y_1\leqslant 1, \quad \quad \lambda\beta (a+b+y)-\beta\leqslant\lambda\beta (a+b+y+y_1)-\beta\leqslant 1,\\
&\lambda\beta(a+b+y+2y_1)-3\beta\leqslant 1, \quad \quad \lambda\beta (a+b+y_1+2y_2)-4\beta\leqslant 1.\\
\endaligned
\right.
\end{equation}
In particular, \eqref{equation:deg-6-D-S-1} is log canonical at
every point of $F_1$ that is different from $Q_1$ by
Lemma~\ref{lemma:log-pull-back}. If $Q_1\neq L^1\cap F_1$ and
$Q_1\neq P_1$, then $\lambda\beta(a+y)=F_1\cdot
(\lambda\beta(aE_1+\Delta^1))>1,$ by
Theorem~\ref{theorem:adjunction}. But $\lambda\beta(a+y)\leqslant
1$, since $a+y\leqslant 2$. This shows that $Q_1=L^1\cap F_1$ or
$Q_1=P_1$. Since $b-a+y\leqslant 1$ and $a+b+y+y_1\leqslant 3$, we
have $b+y\leqslant 2$. So, if $Q_1=L^1\cap F_1$, then
$$
1<\lambda\beta F_1\cdot\Big(bL^1+\Delta^1\Big)=\lambda\beta(b+y)\leqslant 2\lambda\beta\leqslant 1,%
$$
by Theorem~\ref{theorem:adjunction}. If $Q_1=P_1$, then $6=D\cdot
C>\frac{1+4\beta}{\lambda\beta}$ by
\eqref{equation:deg-6-L-tangent-C} and
Theorem~\ref{theorem:blow-ups}. The latter contradicts
$6\lambda\beta\leqslant 1+4\beta$.

We see that $L$ is tangent to $C$ at the point $P$. Then
$x_1=y_1+b$ and
$$
\lambda\leqslant \mathrm{min}\Big\{1,\frac{1+2\beta}{5\beta},
\frac{1}{2\beta}\Big\},
$$
which gives $6\lambda\beta\leqslant 1+3\beta$. Moreover, we have
$a+y+y_1\leqslant 2$, because $2-a-y-y_1=L^2\cdot
\Delta^2\geqslant 0$. Furthermore, since $2L+L_{23}+E_1\sim-K_S$
and $(S, (1-\beta)C+\lambda\beta(2L+L_{23}+E_1)$ is log canonical,
we may assume that $L_{23}\not\subset\mathrm{Supp}(\Delta)$ by
Lemma~\ref{lemma:convexity}. This gives us $b\leqslant 1$, because
$1-b=\Delta\cdot L_{23}\geqslant 0$. Since
$L+L_{12}+L_{13}+2E_1\sim -K_S$ and
$(S,(1-\beta)C+\lambda\beta(L+L_{12}+L_{13}+2E_1))$ is log
canonical, we may assume that
$L_{12}\not\subset\mathrm{Supp}(\Delta)$ or
$L_{13}\not\subset\mathrm{Supp}(\Delta)$  by
Lemma~\ref{lemma:convexity}. If
$L_{12}\not\subset\mathrm{Supp}(\Delta)$, then $1-a=\Delta\cdot
L_{12}\geqslant 0$, which gives $a\leqslant 1$. Similarly, we get
$a\leqslant 1$ if $L_{13}\not\subset\mathrm{Supp}(\Delta)$. Thus,
we have
\begin{equation}
\label{equation:deg-6-last-3} a\leqslant 1, \qquad \qquad
b\leqslant 1, \qquad \qquad b-a+y\leqslant 1, \qquad \qquad
a+y+y_1\leqslant 2,
\end{equation}
which implies that $\lambda\beta (a+b+y)-\beta\leqslant 1$. In
particular, \eqref{equation:deg-6-D-S-1} is log canonical at every
point of $F_1$ that is different from $Q_1$ by
Lemma~\ref{lemma:log-pull-back}. If $Q_1\ne P_1$ and $Q_1\ne
E_1^1\cap F_1$, then $\lambda\beta y=\lambda\beta\Delta^1\cdot
F_1>1$ by Theorem~\ref{theorem:adjunction}. The latter is
impossible, since $\lambda\beta y\leqslant 2\lambda\beta\leqslant
1$ by \eqref{equation:deg-6-last-3}. If $Q_1=E_1^1\cap F_1$, then
$$
1<E_1^1\cdot \Big(\lambda\beta\Delta^1+\big(\lambda\beta(a+b+y)-\beta\big)F_1\Big)=\lambda\beta(1+2a)-\beta%
$$
by Theorem~\ref{theorem:adjunction}. The latter is impossible,
since $\lambda\beta(1+2a)-\beta\leqslant
3\lambda\beta-\beta\leqslant 1$ by \eqref{equation:deg-6-last-3}.
Thus, we see that $Q_1=P_1$.

By \eqref{equation:deg-6-last-3}, one has $a+2b+y+y_1\leqslant 4$.
This implies that $\lambda\beta(a+2b+y+y_1)-2\beta\leqslant 1$.
Then
$$
\Big(S_2,(1-\beta)C^2+\lambda\beta
bL^2+\lambda\beta\Delta^2+\big(\lambda\beta
(a+b+y)-\beta\Big)F_1^2+\big(\lambda\beta
(a+2b+y+y_1)-2\beta\Big)F_2\Big)%
$$
is not log canonical at a unique point $Q_2\in F_2$ by
Lemma~\ref{lemma:log-pull-back}. If $Q_2\not\in L^2\cup F_1^2\cup
C^2$, then $\lambda\beta y_2=\lambda\beta\Delta^2\cdot F_2>1$ by
Theorem~\ref{theorem:adjunction}, which is impossible, since
$\lambda\beta y_2\leqslant 1$ by \eqref{equation:deg-6-last-3}.
Similarly, if $Q_2=F_2\cap L^2$, then $\lambda\beta
(b+y_2)=\lambda\beta(bL^2+\Delta^2)\cdot F_2>1$ by
Theorem~\ref{theorem:adjunction}, which is impossible, because
$b+y_2\leqslant b+y\leqslant 2$ by \eqref{equation:deg-6-last-3}.
If $Q_2=F_2\cap F_1^2$, then
$$
\lambda\beta (y+y_1+a+b)-\beta=\Big(\lambda\beta\Delta^2+\big(\lambda\beta (a+b+y)-\beta\big)F_1^2\Big)\cdot F_2>1%
$$
by Theorem~\ref{theorem:adjunction}, which is impossible, since
$y+y_1+a+b\leqslant 3$ by \eqref{equation:deg-6-last-3}. Then
$Q_2=P_2$.

We have $\lambda\beta(a+2b+y+y_1+y_2)-3\beta\leqslant 1$, since
$a+2b+y+y_1+y_2\leqslant 5$ by \eqref{equation:deg-6-last-3}. Then
$$
\Big(S_3,(1-\beta)C^3+\lambda\beta\Delta^3+\big(\lambda\beta (a+2b+y+y_1)-2\beta\Big)F_2^3+\big(\lambda\beta (a+2b+y+y_1+y_2)-3\beta\Big)F_3\Big).%
$$
is not log canonical at a \emph{unique} point $Q_3\in F_3$ by
Lemma~\ref{lemma:log-pull-back}. If $Q_3\not\in F_2^3\cup C^3$,
then $\lambda\beta y_3=\lambda\beta\Delta^3\cdot F_3>1$ by
Theorem~\ref{theorem:adjunction}, which is impossible, because
$\lambda\beta y_3\leqslant 1$ by \eqref{equation:deg-6-last-3}. If
$Q_3=F_3\cap F_2^3$, then Theorem~\ref{theorem:adjunction} gives
$$
1<F^3_2\cdot\Big(\lambda\beta\Delta^3 +(\lambda\beta
(a+2b+y+y_1+y_2)-3\beta)F_3\Big)=\lambda\beta(a+2b+y+2y_1)-3\beta\leqslant
5\lambda\beta-3\beta,
$$
which is impossible, since $a+2b+y+2y_1\leqslant 5$ by
\eqref{equation:deg-6-last-3}. Thus, we see that $Q_3=P_3$. By
Theorem~\ref{theorem:blow-ups} (iv), we have $6=D\cdot
C\geqslant\mathrm{mult}_{P}(D\cdot
C)>\frac{1+3\beta}{\lambda\beta}$. The latter is impossible, since
we already proved earlier that $6\lambda\beta\leqslant 1+3\beta$.
\end{proof}

\begin{lemma}
\label{lemma:deg-5} Suppose that $K_{S}^2=5$. Then
\eqref{equation:log-pair}  is log canonical at $P$.
\end{lemma}

\begin{proof}
Suppose that \eqref{equation:log-pair} is not log canonical at
$P$. Let us use the notation of \S\ref{subsection:deg-6}. Then
$\lambda=\min\{1,\frac{1}{2\beta}\}$.
 This implies that
$5\lambda\beta\leqslant 1+3\beta$. By
Lemma~\ref{lemma:deg-7-4-blow-ups}, at least one of the conditions
(i), (ii) and (iii) in Lemma~\ref{lemma:deg-7-4-blow-ups} is not
satisfied. In particular, if $a+x\leqslant 2$, then
$\lambda\beta(a+x+2x_1)-3\beta>1$.

Without loss of generality, we may assume that
$\mathcal{L}=L_{12}$. Let $B_3$ be the proper transform on $S$ of
the line in $\mathbb{P}^2$ that passes through $\pi(P)$ and
$\pi(E_3)$, and let $B_4$ be the proper transform on $S$ of the
line in $\mathbb{P}^2$ that passes through $\pi(P)$ and
$\pi(E_4)$. Since $L_{12}+B_3+B_4\sim -K_{S}$ and
$(S,(1-\beta)C+\lambda\beta (L_{12}+B_3+B_4))$ is log canonical,
we may assume that at least one curve among $B_3$ and $B_4$ is not
contained in $\mathrm{Supp}(\Omega)$. Intersecting this curve with
$\Omega$, we get $a+x\leqslant 2$. Then
$\lambda\beta(a+x+2x_1)-3\beta>1$. This implies that $a+x+2x_1>5$.

Denote the proper transform of the curve $B_3$ on the surface
$S_1$ by $B_3^1$, and denote the proper transform of the curve
$B_4$ on the surface $S_1$ by $B_4^1$. Recall $P_1=C^1\cap F_1$.

Suppose that $P_1\not\in B_3^1\cup B_4^1$. Then $B_3$ and $B_4$ do
not tangent $C$ at $P$. Let $R$ be the proper transform on $S$ of
the line in $\mathbb{P}^2$ that is tangent to $\pi(C)$ at the
point $\pi(P)$, let $R_1$ be the proper transform on $S$ of the
conic in $\mathbb{P}^2$ that tangents to $\pi(C)$ at the point
$\pi(P)$ and passes through the points $\pi(E_2)$, $\pi(E_3)$ and
$\pi(E_4)$, and let $R_2$ be the proper transform on $S$ of the
conic in $\mathbb{P}^2$ that tangents to $\pi(C)$ at the point
$\pi(P)$ and passes through the points $\pi(E_1)$, $\pi(E_3)$ and
$\pi(E_4)$. Since $P_1\not\in B_3^1\cup B_4^1$, the curves $R_1$
and $R_2$ are irreducible. Hence
$\frac{1}{2}L_{12}+\frac{1}{2}R+\frac{1}{2}R_1+\frac{1}{2}R_2\sim_{\mathbb{R}}
-K_S$ and $(S,(1-\beta)C + \lambda\beta
(\frac{1}{2}L_{12}+\frac{1}{2}R+\frac{1}{2}R_1+\frac{1}{2}R_2))$
is log canonical. By Lemma~\ref{lemma:convexity}, we may assume
that one curve among $R$, $R_1$ and $R_2$ is not contained in
$\mathrm{Supp}(D)$. Denote this curve by $Z$, and denote its
proper transform on $S_1$ by $Z^1$. Then $P_1\in Z^1$ and
$3-a-x=Z^1\cdot \Omega^1\geqslant x_1$, which is impossible, since
$a+x\leqslant 2$ and $a+x+2x_1>5$.

We see that $P_1=B_3^1\cap F_1$ or $P_1=B_4^1\cap F_1$. Without
loss of generality, we may assume that $P_1=B_3^1\cap F_1$. Then
$B_3\subset\mathrm{Supp}(\Omega)$, since otherwise we would have
$2-a-x=B_3^1\cdot \Omega^1\geqslant x_1$, which is impossible,
since $a+x\leqslant 2$. We write $\Omega=bB_3+\Delta$, where $b\in
\mathbb R_{>0}$ and $\Delta$ is an effective $\mathbb{R}$-divisor
on $S$ such that $B_3\not\subset\mathrm{Supp}(\Delta)$. Denote
the proper transform of the divisor $\Delta$ on $S_1$ by
$\Delta^1$. Let $y=\mathrm{mult}_{P}(\Delta)$ and
$y_1=\mathrm{mult}_{P_1}(\Delta^1)$. Then $x=b+y$ and $x_1=b+y_1$.
%But
%$2-a-y_0=B_3^1\cdot\Delta^1\geqslant y_1.$
We have $b-a+y\leqslant 1$ by Corollary~\ref{corollary:x-a-1} and $a+b+y=a+x\leqslant 2$, which implies a contradiction
$a+x+2x_1\leqslant 2+2y+2b\leqslant 5$.
\end{proof}

\begin{lemma}
\label{lemma:deg-4-line} Suppose that $K_{S}^2=4$. Then
\eqref{equation:log-pair} is log canonical at $P$.
\end{lemma}

\begin{proof}
Suppose that \eqref{equation:log-pair} is not log canonical at
$P$. Let us use the notation \S\ref{subsection:deg-4}. Then
$\lambda\beta<\frac{2}{3}$. Without loss of generality, we may
assume that $P\in E$. Then $P=E\cap C$. By
Lemma~\ref{lemma:P-in-two-lines}, the point $P$ is not contained
in any other $(-1)$-curve. By Lemma~\ref{lemma:P-in-C}, we have
$E\subset\mathrm{Supp}(D)$.

The log pair
$(S,(1-\beta)C+\lambda\beta(\frac{3}{2}E+\frac{1}{2}(E_1+E_2+E_3+E_4+E_5)))$
is log canonical and
$\frac{3}{2}E+\frac{1}{2}(E_1+E_2+E_3+E_4+E_5)\sim_{\mathbb{R}}
-K_{S}$. By Lemma~\ref{lemma:convexity}, we may assume that
$\mathrm{Supp}(\Omega)$ does not contain one curve among $E_1$,
$E_2$, $E_3$, $E_4$, $E_5$. Intersecting this curve with $\Omega$,
we get $a\leqslant 1$. Let $L_1$, $L_2$, $L_3$, $L_4$, $L_5$ be
the proper transforms on $S$ of the lines in $\mathbb{P}^2$ that
pass through $\pi(P)$ and $\pi(E_1)$, $\pi(E_2)$, $\pi(E_3)$,
$\pi(E_4)$, $\pi(E_5)$, respectively. Then
$\frac{2}{3}E+\frac{1}{3}(L_1+L_2+L_3+L_4+L_5)\sim_{\mathbb{R}}
-K_{S}$, and
$(S,(1-\beta)C+\lambda\beta(\frac{2}{3}E+\frac{1}{3}(L_1+L_2+L_3+L_4+L_5)))$
is log canonical. By Lemma~\ref{lemma:convexity}, we may assume
that $\mathrm{Supp}(\Omega)$ does not contain  one curve among
$L_1$, $L_2$, $L_3$, $L_4$, $L_5$. Intersecting this curve with
$\Omega$, we get $a+x\leqslant 2$. Recall that $a\leqslant 1$ by
Corollary~\ref{corollary:x-a-1}. Thus, we have
\begin{equation}
\label{equation:deg-4-main-inequalities}
a\leqslant 1, \qquad \qquad \qquad x-a\leqslant 1,\qquad \qquad \qquad a+x\leqslant 2,\\
\end{equation}
which implies that $x\leqslant\frac{3}{2}$ and
$\lambda\beta(a+x)-\beta\leqslant 1$. In particular, we have
$\lambda \beta x\leqslant 1$.

Denote the proper transform of the curve $E$ on $S_1$  by $E^1$.
Then $\lambda\beta(a+x)-\beta\leqslant 1$, since $a+x\leqslant 2$.
Thus, the log pair $(S_1, (1-\beta)C^1+\lambda\beta a
E^1+\lambda\beta \Omega^1+(\lambda\beta(a+x)-\beta)F_1)$ is not
log canonical at the unique point $Q_1\in F_1$ by
Lemma~\ref{lemma:log-pull-back}. Note that
$\lambda\beta(a+x)-\beta>0$ by Lemma~\ref{lemma:Skoda}. Moreover,
either $Q_1=P_1$ or $Q_1=E^1\cap F_1$, since otherwise we would
have $\lambda x=\lambda\beta\Omega^1\cdot F_1>1$ by
Theorem~\ref{theorem:adjunction}. If $Q_1=E^1\cap F_1$, then
Theorem~\ref{theorem:Trento} implies
$$
\lambda\beta(1+a-x)=\lambda\beta\Omega^1\cdot E^1>2\Big(1+\beta-\lambda\beta(x+a)\Big)%
$$
or $\lambda\beta x=\lambda\beta\Omega^1\cdot F_1>2(1-\lambda\beta
a)$. The former inequality gives $\lambda\beta(1+3a+x)>2+2\beta$,
which is impossible since
 $1+3a+x\leqslant 5$ by
\eqref{equation:deg-4-main-inequalities}. The latter inequality
gives that $\lambda\beta(x+2a)>2$, which is impossible since
$x+2a\leqslant 3$ by \eqref{equation:deg-4-main-inequalities}.
Thus, we see that $Q_1=P_1$.

Let $R$ be the proper transform on $S$ of a line in $\mathbb{P}^2$
that is tangent to $\pi(C)$ at the point $\pi(P)$. Then either
$-K_{S}\cdot R=3$ or  $-K_{S}\cdot R=2$. Moreover, $-K_{S}\cdot
R=3$ if and only if $\pi(R)$ does not contain any of the points
$\pi(E_1)$, $\pi(E_2)$, $\pi(E_3)$, $\pi(E_4)$, $\pi(E_5)$.

Suppose that $-K_{S}\cdot R=2$. Without loss of generality, we may
assume that $R=L_1$. We write $\Omega=bL_1+\Delta$, where $b$ is a
non-negative real number, and $\Delta$ is an effective
$\mathbb{R}$-divisor on $S$ whose support does not contain the
curve $L_1$. Denote the proper transform of the curve $L_1$ on
$S_1$ by $L_1^1$, and denote the proper transform of $\Delta$ on
$S_1$  by $\Delta^1$. Let $y=\mathrm{mult}_{P}(\Delta)$ and
$y_1=\mathrm{mult}_{P_1}(\Delta^1)$. Then $x=y+b$. Since
$(S,(1-\beta)C+\lambda\beta(E+E_1+L_1))$ is log canonical and
$E+E_1+L_1\sim -K_{S}$, we may assume that $b=0$ or
$\mathrm{Supp}(\Delta)$ does not contain $E_1$ by
Lemma~\ref{lemma:convexity}. Thus, if $b\ne 0$, then
$1-a-b=\Delta\cdot E_1\geqslant 0$. With
\eqref{equation:deg-4-main-inequalities}, this gives
$y+2b\leqslant2$
 and $2+a+y+2b\leqslant\frac{9}{2}$.
On the other hand, we have $2-a-y=\Delta^1\cdot L_1^1\geqslant
y_1$, which implies that $a+2y_1\leqslant 2$, since $y\geqslant
y_1$. Thus, we see that $y_1\leqslant 1$. Then
$\mathrm{mult}_{P_1}((1-\beta)C^1+\lambda\beta\Delta^1)=1-\beta+\lambda\beta
y_1\leqslant 1$. Applying Theorem~\ref{theorem:Trento}, we see
that
$$
1-\beta+\lambda\beta(2-a-y)=\Big((1-\beta)C^1+\lambda\beta\Delta^1\Big)\cdot L_{1}^1>2\Big(1+\beta-\lambda\beta(a+b+y)\Big)%
$$
or $1-\beta+\lambda\beta
y=((1-\beta)C^1+\lambda\beta\Delta^1)\cdot F_1>2(1-\lambda\beta
b)$. This gives $\lambda\beta(2+a+y+2b)>1+3\beta$ or
$\lambda\beta(y+2b)>1+\beta$. The former inequality  is
impossible, because $2+a+y+2b\leqslant\frac{9}{2}$. The latter
inequality is also impossible, because $y+2b\leqslant2$.

We have $-K_{S}\cdot R=3$. Then $R$ is irreducible and $R+E\sim
-K_{S}$. Since $(S,(1-\beta)C+\lambda\beta(R+E))$ is log
canonical, we may assume that $R\not\subset\mathrm{Supp}(\Omega)$
by Lemma~\ref{lemma:convexity}. Denote the proper transform of the
curve $R$ on the surface $S_1$ by $R^1$. Then
$3-2a-x=\Omega^1\cdot R^1\geqslant x_1$, which gives
$x+x_1+2a\leqslant 3$. Then $\lambda\beta(a+x+x_1)-2\beta\leqslant
1$ by \eqref{equation:deg-4-main-inequalities}. Thus, the log pair
$$
\Big(S_2,(1-\beta)C^2+\lambda\beta\Omega^2+\big(\lambda\beta(a+x)-\beta\big)F_1^2+\big(\lambda\beta(a+x+x_1)-2\beta\big)F_2\Big)%
$$
is not log canonical at a unique point $Q_2\in F_2$ by
Lemma~\ref{lemma:log-pull-back}. Note that
$\lambda\beta(a+x+x_1)-2\beta>0$ by Lemma~\ref{lemma:Skoda}. If
$Q_2\neq P_2$ and $Q_2\neq F_1^2\cap F_2$, then
Theorem~\ref{theorem:adjunction} gives $\lambda\beta
x_1=\lambda\beta\Omega^2\cdot F_2>1$, which is impossible, since
$\lambda\beta x_1\leqslant\lambda \beta x\leqslant 1$ by
\eqref{equation:deg-4-main-inequalities}. If $Q_2=F^2_1\cap F_2$,
then Theorem~\ref{theorem:adjunction} gives
$$
\lambda\beta(a+2x)-2\beta\geqslant\Big(\lambda\beta\Omega^2+(\lambda\beta(a+x+x_1)-2\beta)F_2\Big)\cdot F_1^2>1%
$$
which is impossible, since $a+2x\leqslant \frac{7}{2}$, by
\eqref{equation:deg-4-main-inequalities}. Hence, we see that
$Q_2=P_2$.

One has $\lambda\beta(a+x+x_1+x_2)-3\beta\leqslant 1$ by
\eqref{equation:deg-4-main-inequalities}, since $x+x_1+2a\leqslant
3$ and $x_2\leqslant x_1\leqslant x$. Thus, it follows from
Lemma~\ref{lemma:log-pull-back} that
$$
\Big(S_3,(1-\beta)C^3+\lambda\beta\Omega^3+\big(\lambda\beta(a+x+x_1)-2\beta\big)F_2^3+\big(\lambda\beta(a+x+x_1+x_2)-3\beta\big)F_3\Big)
$$
is not log canonical at a unique point $Q_3\in F_3$. Note that
$\lambda\beta(a+x+x_1+x_2)-3\beta>0$ by Lemma~\ref{lemma:Skoda}.
If $Q_3\neq P_3$ and $Q_3\neq F_2^3\cap F_3$, then
Theorem~\ref{theorem:adjunction} gives $\lambda\beta
x_2=\lambda\beta\Omega^3\cdot F_3>1$, which is impossible, since
$\lambda\beta x_2\leqslant\lambda \beta x\leqslant 1$ by
\eqref{equation:deg-4-main-inequalities}. If $Q_3=F^3_2\cap F_3$,
then  Theorem~\ref{theorem:adjunction} gives
$$
\lambda\beta(a+x+2x_1)-3\beta=\Big(\lambda\beta\Omega^3+(\lambda\beta(a+x+x_1+x_2)-3\beta)F_3\Big)\cdot F_2^3>1%
$$
which contradicts \eqref{equation:deg-4-main-inequalities}, since
$x+x_1+2a\leqslant 3$. Thus, we have $Q_3=P_3$. Then
Theorem~\ref{theorem:adjunction} gives
$$
\beta\geqslant
4\lambda\beta-3\beta=C^3\cdot\Big(\lambda\beta\Omega^3+(\lambda\beta(a+x+x_1+x_2)-3\beta)F_3\Big)>1,
$$
which is impossible, since $\beta\in(0,1]$.
\end{proof}

This completes the proof of Lemma~\ref{lemma:P-in-one-line}.

\end{document}